\documentclass[11pt, a4paper]{article}
\usepackage{amsfonts, amssymb, amsmath, amsthm}
\usepackage{esint}
\usepackage{latexsym}
\usepackage{mathrsfs}
\usepackage[margin=1in]{geometry}
\usepackage[colorlinks, linkcolor=blue, anchorcolor=blue, citecolor=blue]{hyperref}
\usepackage{graphicx}
\usepackage{enumerate}

\allowdisplaybreaks 

\newtheorem{theorem}{Theorem}[section]
\newtheorem{coro}{Corollary}[section]
\newtheorem{prop}{Proposition}[section]
\newtheorem{definition}{Definition}[section]
\newtheorem{lemma}{Lemma}[section]
\newtheorem{remark}{Remark}[section]
\numberwithin{equation}{section}

\date{}

\begin{document}

\title{Convergence rates in almost-periodic homogenization of higher-order elliptic systems}
\author{Yao Xu \thanks{Supported by China Postdoctoral Science Foundation (2019TQ0339).} \and Weisheng Niu \thanks{Corresponding author. Supported by the NSF of China (11701002) and the NSF of Anhui Province (1708085MA02).}}


\maketitle
\pagestyle{plain}
\begin{abstract}
This paper concentrates on the quantitative homogenization of higher-order elliptic systems with almost-periodic coefficients in bounded Lipschitz domains. For almost-periodic coefficients in the sense of H. Weyl, we establish uniform local $L^2$ estimates for the approximate correctors. Under an additional assumption \eqref{estimate_condition_rho} on the frequencies of the coefficients, we derive the existence of true correctors as well as the $O(\varepsilon)$ convergence rate in $H^{m-1}$. As a byproduct, the large-scale H\"{o}lder estimate and a Liouville theorem are obtained for higher-order elliptic systems with almost-periodic coefficients in the sense of Besicovitch. Since \eqref{estimate_condition_rho} is not well-defined for equivalence classes of almost-periodic functions in the sense of H. Weyl or Besicovitch, we provide another condition yielding the $O(\varepsilon)$ convergence rate under perturbations of the coefficients.
\end{abstract}

\section{Introduction}

Let $\Omega$ be a bounded Lipschitz domain in $\mathbb{R}^d$. We consider the quantitative homogenization for the $2m$-order elliptic system with almost-periodic (a.p.) coefficients
\begin{equation}\label{intro_eq1}
\begin{cases}
\mathcal{L}_\varepsilon u_\varepsilon =f  &\text{ in } \Omega,\\
Tr (D^\gamma u_\varepsilon)=g_\gamma  & \text{ on } \partial\Omega, \text{ for  } 0\leq|\gamma|\leq m-1,
\end{cases}
\end{equation}
where $u_\varepsilon: \Omega \rightarrow \mathbb{R}^n$ is a vector function, \begin{equation}\label{intro_exp_L} \mathcal{L}_\varepsilon=\mathcal{L}_\varepsilon^{A}= (-1)^{m}\sum_{|\alpha|=|\beta|=m} D^\alpha(A_{ij}^{\alpha\beta}(\frac{x}{\varepsilon})D^\beta),\end{equation}
$1\leq i, j\leq n$, $\alpha, \beta,\gamma$ are multi-indexes with components $\alpha_k, \beta_k,\gamma_k, k=1,2,...,d$, and
$$ |\alpha|=\sum_{k=1}^d \alpha_k, ~~D^\alpha=D_{x_1}^{\alpha_1} D_{x_2}^{\alpha_2}\cdot\cdot\cdot D_{x_d}^{\alpha_d}. $$
Assume that the coefficient matrix $A(y)=(A_{ij}^{\alpha \beta}(y))$ is real, bounded measurable with
\begin{align}\label{intro_cond_bdd}
\|A^{\alpha \beta}_{ij}(y)\|_{L^\infty(\mathbb{R}^d)}\leq 1/\mu,
\end{align}
and satisfies the coercivity condition (the summation convention for $i, j$ is  used)
\begin{align}\label{intro_cond_el}
\sum_{|\alpha|=|\beta|=m} \int_{\mathbb{R}^d} A_{ij}^{\alpha\beta} D^\beta\phi_j D^\alpha\phi_i \geq \mu \sum_{|\alpha|=m}\|D^\alpha\phi\|^2_{L^2(\mathbb{R}^d)} ~\textrm{ for any }~ \phi\in C_c^\infty(\mathbb{R}^d; \mathbb{R}^n),
\end{align}
where $\mu>0$. We further assume that $A$ is a.p. in the sense of Besicovitch, i.e. $A\in B^2$ (see Section \ref{hom_sec_almost}).

Let $W\!A^{m,p}(\partial\Omega; \mathbb{R}^n)$ be the Whitney-Sobolev space composed of $\dot{g}=  \{g_\alpha\}_{|\alpha|\leq m-1}$, that is,
the completion of the set $\left\{\{ D^\alpha \mathscr{G}\mid_{\partial\Omega}\}_{|\alpha|\leq m-1}:  \mathscr{G}\in C_c^\infty(\mathbb{R}^d; \mathbb{R}^n ) \right\}$
with respect to norm
$$ \| \dot{g} \|_{W\!A^{m,p}(\partial\Omega)} =\sum_{|\alpha|\leq m-1} \|g_\alpha\|_{L^p(\partial \Omega)} + \sum_{|\alpha|=m-1} \|\nabla_{tan} D^\alpha g_\alpha\|_{L^p(\partial \Omega)}.$$
It is known that, for any $\dot{g}\in W\!A^{m, 2}(\partial\Omega; \mathbb{R}^n)$ and $f\in H^{-m}(\Omega;\mathbb{R}^n)$, problem (\ref{intro_eq1}) admits a unique weak solution $u_\varepsilon\in H^m(\Omega;\mathbb{R}^n)$ such that
\begin{gather*}
\int_\Omega \sum_{|\alpha|=|\beta|=m} D^\alpha v_i A_{ij}^{\alpha \beta}(x/\varepsilon)D^\beta u_{\varepsilon, j} dx= \langle f, v\rangle, ~~\forall~v\in H_0^m(\Omega;\mathbb{R}^n ),\\
\|u_\varepsilon\|_{H^m(\Omega)} \leq C \left\{\|f\|_{H^{-m}(\Omega)} +\|\dot{g}\|_{W\!A^{m,2}(\partial\Omega)}\right\}.
\end{gather*}
Thanks to the qualitative homogenization result in Section \ref{sec_hom}, $u_\varepsilon$ converges weakly in $H^m(\Omega; \mathbb{R}^n)$ to $u_0$, solving
\begin{equation}\label{intro_eq_hom}
\begin{cases}
\mathcal{L}_0 u_0 =f  &\text{ in } \Omega,\\
Tr (D^\gamma u_0)=g_\gamma  & \text{ on } \partial\Omega, \text{ for  } 0\leq|\gamma|\leq m-1,
\end{cases}
\end{equation}
where $\mathcal{L}_0$ is a $2m$-order elliptic operator with constant coefficients depending only on $A$. The primary aim of this paper is to establish the convergence rate of $u_\varepsilon$ to $u_0$.

Homogenization of elliptic equations or systems with a.p. coefficients goes back to \cite{Kozlov1979}, where the qualitative result was obtained for second-order elliptic operators with a.p. coefficients, as well as the sharp convergence rate in $C(\overline{\Omega})$ for operators with sufficiently smooth quasiperiodic coefficients under proper assumptions on the spectrum of $A$. Afterwards, homogenization of linear or nonlinear operators involving a.p. coefficients was further studied by many authors in different contexts (see e.g. \cite{Dungey2001, Bondarenko2005, Lions2005, Caffarelli2010}). Recently, in \cite{Shen2015_Convergence} Z. Shen investigated uniform H\"{o}lder estimates and the convergence rate for second-order elliptic systems with uniformly a.p. coefficients. Based on the convergence rates, the uniform interior and boundary Lipschitz estimates were then obtained by S.N. Armstrong and Z. Shen in \cite{Armstrong2016_Lipschitz}. Further investigation was then carried out by S.N. Armstrong et al. in \cite{Armstrong2016_Bounded}, where they derived the uniform boundedness of approximate correctors and the existence of true correctors using a brilliant quantitative ergodic theorem. More recently, Z. Shen and J. Zhuge in \cite{Shen2018_Approximate, Zhuge2017_Uniform} conducted a very comprehensive study on the homogenization of second-order elliptic systems with a.p. coefficients in the sense of H. Weyl (denoted by $APW^2$), a broader class of a.p. functions than uniformly a.p. functions. Under proper assumptions, the sharp convergence rate, the existence of true correctors as well as the uniform Lipschitz estimate were established.

Quantitative homogenization for higher-order elliptic equations, even in the periodic case, is less understood until very recently. In \cite{Kukushkin2016, Pastukhova2016, Pastukhova2017}, the $O(\varepsilon)$ convergence rate in $L^2$ was established for higher-order elliptic equations with periodic coefficients on the whole space, while in \cite{Suslina2017_Dirichlet, Suslina2018_Neumann} similar results were obtained for more general higher-order systems with Dirichlet or Neumann boundary data in bounded $C^{2m}$ domains. More recently, in \cite{Niu2018_Convergence} Z. Shen and the authors of the present paper investigated the convergence rate in periodic homogenization of higher-order elliptic systems with symmetric coefficients in Lipschitz domains, where the $O(\varepsilon)$ convergence rate in $W^{m-1, 2d/(d-1)}$ and some uniform interior estimates were obtained. Without the symmetry assumption, a suboptimal convergence rate, as well as some uniform boundary estimates, was obtained in \cite{Niu2019_Boundary}. By now, very little is known about quantitative homogenization of higher-order elliptic equations or systems with a.p. coefficients. This motivates the study of the present paper.

Compared to the periodic setting, the main obstacle in the study of quantitative homogenization in a.p. setting is that the equations for correctors may not be solvable. Therefore, following the idea in \cite{Shen2018_Approximate}, we introduce the so-called approximate corrector $\chi_T$ by
\begin{equation*}
\mathcal{L}_1 \chi_T+T^{-2m}\chi_T=-\mathcal{L}_1 P~\quad\text{in}~\mathbb{R}^d,
\end{equation*}
where $T>0$ and $P$ is a monomial of degree $m$ (see Section \ref{sec_appcor}). Our first result concerns the uniform estimates on $\chi_T$, which plays an essential role in the study of convergence rate.

\begin{theorem}\label{estimate_thm_1}
Suppose that $A\in APW^2$ and satisfies \eqref{intro_cond_bdd}--\eqref{intro_cond_el}. Fix $k\geq 1$ and $\sigma\in(0, 1)$. Then for any $T\geq 2$,
\begin{gather}
\|\nabla^m \chi_T\|_{S^2_1}\leq C_\sigma T^\sigma,\label{estimate_es_m_chi}\\
\|\nabla^l\chi_T\|_{S^2_1}\leq C_\sigma\int_1^T t^{m-l-1}\inf_{1\leq L\leq t}\bigg\{\rho_k(L, t)+\exp\Big(-\frac{ct^2}{L^2}\Big)\bigg\}\Big(\frac{T}{t}\Big)^\sigma dt,\label{estimate_es_chi_S1}
\end{gather}
$0\leq l\leq m-1$, where $\nabla^l u$ denotes the vector $(D^\alpha u)_{|\alpha|=l}$, $C_\sigma$ depends only on $d, m, n, k, \sigma$ and $A$, and $c$ depends only on $d$ and $k$.
\end{theorem}

We mention that in Theorem \ref{estimate_thm_1} we have used the notation
\begin{align*}\|u\|_{S^p_R}:=\sup_{x\in \mathbb{R}^d}\Big(\fint_{B(x, R)}|u|^p\Big)^{1/p},
\end{align*}
and the quantity $\rho_k(L, R)$ measuring the frequencies of $A$ defined by \eqref{def_rho}. Our next two theorems provide the existence of true correctors and the optimal convergence rate under proper assumptions on $A$.

\begin{theorem}\label{estimate_thm_2}
Suppose that $A\in APW^2$ and satisfies \eqref{intro_cond_bdd}--\eqref{intro_cond_el}. Suppose further that there exist some $k\geq1$ and $\theta>m$ such that
\begin{align}\label{estimate_condition_rho}
\rho_k(L, L)\leq CL^{-\theta} \quad\textrm{for any } L\geq 1.
\end{align}
Then $\|\nabla^l \chi_T\|_{S^2_1}\leq C$ for each $0\leq l\leq m$. Moreover, for each $P^\gamma$, the system for the (true) corrector
$$\mathcal{L}_1 u=-\mathcal{L}_1(P^\gamma)\textrm{~in~} \mathbb{R}^d$$ has a weak solution $\chi^\gamma$ such that $\nabla^l \chi^\gamma\in APW^2$ for each $0\leq l\leq m$.
\end{theorem}

\begin{theorem}\label{conver_thm_2}
Assume that $\Omega$ is a bounded Lipschitz domain and $A\in APW^2$ satisfies \eqref{intro_cond_bdd}--\eqref{intro_cond_el}. Let $u_\varepsilon$ be the weak solution to Dirichlet problem \eqref{intro_eq1} and $u_0$ be the weak solution to the homogenized problem \eqref{intro_eq_hom}. Suppose further $A=A^*$ if $n\geq 2$ and $u_0\in H^{m+1}(\Omega; \mathbb{R}^n)$. Then for any $0<\varepsilon<1$ and $T=\varepsilon^{-1/m}$,
\begin{align}
&\quad\|u_\varepsilon-u_0\|_{H^{m-1}_0(\Omega)}\nonumber\\&\leq C_\sigma\Big(\sum_{l\leq m-1}\Theta_{k, l, \sigma}(T)\Big)\Big\{\|\nabla^m \chi_T-\psi\|_{B^2}+T^{-m}\sum_{l\leq m-1}\Theta_{k, l, \sigma}(T)\Big\}\|u_0\|_{H^{m+1}(\Omega)},\label{conver_es_rate1}
\end{align}
where $\Theta_{k, l, \sigma}(T)$ denotes the integral in the r.h.s. of \eqref{estimate_es_chi_S1}, $\psi$ is the unique solution to \eqref{hom_eq_corrector}. Furthermore, if \eqref{estimate_condition_rho} holds for some $\theta>m$ and $k\geq 1$, then
\begin{align}
\|u_\varepsilon-u_0\|_{H^{m-1}_0(\Omega)}\leq C\varepsilon\|u_0\|_{H^{m+1}(\Omega)}.\label{intro_es_converrate}
\end{align}
\end{theorem}
Theorem \ref{estimate_thm_2} above extends the results of Theorem 1.2 in \cite{Shen2018_Approximate} to higher-order elliptic systems, while Theorem \ref{conver_thm_2} above generalizes the corresponding results of Theorem 1.4 in \cite{Shen2018_Approximate}, where similar results were obtained for second-order elliptic systems in $C^{1,1}$ domains. If $\partial\Omega \in C^{m, 1}$, estimate \eqref{intro_es_converrate} in Theorem \ref{conver_thm_2} still holds with the symmetry assumption of $A$ removed. Finally, we remark that, without the additional assumption \eqref{estimate_condition_rho}, the true correctors may not exist and the $O(\varepsilon)$ convergence rate is not always satisfied even if the coefficient $A$ is very smooth (see e.g. \cite{Bondarenko2005}).

  It is known that elements of $B^2$ and $APW^2$ are equivalence classes under the equivalent relation induced by their semi-norms. However, the quantity $\rho_k(L, R)$ is not well-defined for the equivalence class of $A$ in $B^2$ or $APW^2$. More precisely, condition \eqref{estimate_condition_rho} may fail for some function while it holds for another one in the same equivalence class. As a supplement of \eqref{estimate_condition_rho}, we provide here a sufficient condition for the sharp convergence rate
in terms of perturbations on the coefficients.

\begin{definition}
Let $\Omega$ be a bounded Lipschitz domain and $A\in B^2$ satisfy \eqref{intro_cond_bdd}--\eqref{intro_cond_el}. We say $A$ has the $O(\varepsilon)$-convergence property if, for any $\dot{g}\in W\!A^{m, 2}(\partial\Omega; \mathbb{R}^n)$ and $f\in H^{-m}(\Omega;\mathbb{R}^n)$,
\begin{align}
\|u_\varepsilon-u_0\|_{H^{m-1}_0(\Omega)}\leq C\varepsilon\|u_0\|_{H^{m+1}(\Omega)},
\end{align}
where $C$ is independent of $\varepsilon$, $f$ and $\dot{g}$, $u_\varepsilon$ is the weak solution to problem \eqref{intro_eq1} with coefficients $A$ and $u_0$ is the solution to the corresponding homogenized problem \eqref{intro_eq_hom} with $u_0\in H^{m+1}(\Omega)$.
\end{definition}

\begin{theorem}\label{intro_thm_4}
Assume that $\Omega$ is a bounded Lipschitz domain and $A, \widetilde{A}\in B^2$ satisfy \eqref{intro_cond_bdd}--\eqref{intro_cond_el}. Suppose that $A$ has the $O(\varepsilon)$-convergence property. Then there exists  $p>2$, depending only on $m, n, \Omega, \mu$, such that if
\begin{align}
\Big(\fint_{B(0, T)}|A-\widetilde{A}|^p dx\Big)^{{1/p}}\leq CT^{-1} ~\textrm{ for any }~ T\geq 1,\label{pertur_cond_pertur}
\end{align}
then $\widetilde{A}$ also has the $O(\varepsilon)$-convergence property.
\end{theorem}

Now we present the outline and key ideas of the paper. We start in Section \ref{sec_almost} with a brief review on a.p. functions, along with the qualitative homogenization theory for higher-order systems with a.p. coefficients. Two compactness results, namely Theorems \ref{hom_thm_compactness} and \ref{hom_thm_compactness2}, on a sequence of elliptic operators $\mathcal{L}^{A_l}_{\varepsilon_l}+\lambda_l$ are established, where each $A_l$ is a translation of $A$. Theorem \ref{hom_thm_compactness}, involving bounded translations of $A$ only, is proved by Tartar's method of test functions as \cite{Jikov1994}, while Theorem \ref{hom_thm_compactness2}, allowing arbitrary translations of $A$, follows from Theorem \ref{hom_thm_compactness} and a perturbation argument.

In Section \ref{sec_pre}, we provide several useful lemmas for higher-order elliptic systems, including a Poincar\'{e}-Sobolev lemma, Caccioppoli's inequalities and Meyers' reverse H\"{o}lder inequalities, while in Section \ref{sec_appcor}, we introduce the approximate corrector $\chi_T$ and establish some elementary estimates.

In Section \ref{sec_holder}, by using a compactness argument introduced by M. Avellaneda and F. Lin in \cite{Avellaneda1987_Compactness}, we prove a large-scale H\"{o}lder estimate for higher-order elliptic system with $B^2$-coefficients, from which some principal uniform estimates on $\chi_T$ are derived. To adapt the compactness argument to higher-order elliptic systems, we take full advantage of Theorem \ref{hom_thm_compactness2} and the Poincar\'{e}-Sobolev lemma in Section \ref{sec_pre}. Meanwhile, we derive a Liouville theorem for higher-order elliptic systems with $B^2$-coefficients using Theorem \ref{hom_thm_compactness} and the compactness argument. We mention that the large-scale H\"{o}lder estimate and the Liouville theorem are comparable to the results for second-order elliptic systems in \cite{Shen2018_Approximate}, as our settings of the coefficients are more general, belonging to $B^2$ rather than $APW^2$.

In Section \ref{sec_ergodic}, we extend the quantitative ergodic theorem in \cite{Shen2018_Approximate} to the higher-order case, which allows us to bound the spatial averages of a uniformly locally integrable function by its higher-order differences and its $S^2_R$-norm with exponential decay. Thanks to the large-scale H\"{o}lder estimates, the higher-order differences of $\chi_T$ can then be controlled by $\rho_k(L, R)$. Following this idea, in Section \ref{sec_estimate} we provide the proof of Theorem \ref{estimate_thm_1}. Furthermore, by showing that $\{\chi_T\}_{T>0}$ is a Cauchy sequence with respect to $S^2_1$-norm, we obtain the existence of true correctors in $APW^2$ in Theorem \ref{estimate_thm_2}.
 In Section \ref{sec_dual}, we establish estimates for the so-called dual approximate correctors $\phi_T$.

 Finally, Section \ref{sec_conver} is devoted to the proofs of Theorems \ref{conver_thm_2} and \ref{intro_thm_4}. Based on the estimates for $\chi_T$ and $\phi_T$ aforementioned, Theorem  \ref{conver_thm_2} follows form the duality argument inspired by \cite{Suslina2013_Dirichlet} (see also \cite{Suslina2013_Neumann,Shen2017_Boundary,Shen2017_Convergence}). Theorem \ref{intro_thm_4} is proved by estimating the difference between $u_\varepsilon$ and $\widetilde{u}_\varepsilon$, the solutions corresponding to $A$ and $\widetilde{A}$ respectively, since $u_\varepsilon$ and $\widetilde{u}_\varepsilon$ have the same limit $u_0$ under condition \eqref{pertur_cond_pertur}.

Throughout this paper, $C$ will denote positive constants, which may depend on $m, n, \mu$ and $\Omega$, but never on $\varepsilon$ or $T$, and would differ from each other even in the same line. The usual summation convention for repeated indices will be  used henceforth.  Moreover, if it is clear to understand, we may use scalar notation for concision, that is, omitting the subscripts $i, j$. We also use notations $\fint_E f:=(1/|E|)\int_E f$ for the integral average of $f$ over $E$ and $x^\alpha$ for the monomial $x_1^{\alpha_1}\cdots x_d^{\alpha_d}$.


\section{Almost-periodic functions and qualitative homogenization}\label{sec_almost}

\subsection{Almost-periodic functions}\label{hom_sec_almost}

In this part, we provide a brief review on a.p. functions. Let $\rm{Trig}(\mathbb{R}^d)$ denote the set of real trigonometric polynomials on $\mathbb{R}^d$ and $1\leq p<\infty$. For $f\in L^p_{loc}(\mathbb{R}^d)$ and $R>0$, we denote
\begin{gather*}\|f\|_{S^p_R}:=\sup\limits_{x\in\mathbb{R}^d}\Big(\fint_{B(x, R)}|f|^p\Big)^{1/p}, \\
\|f\|_{W^p}:=\limsup\limits_{R\rightarrow \infty}\|f\|_{S^p_R}~~\textrm{ and }~~\| f\|_{B^p}:=\limsup\limits_{R\rightarrow\infty}\Big(\fint_{B(0, R)}|f|^p\Big)^{1/p}.
\end{gather*}
A function $f$ in $L^p_{loc}(\mathbb{R}^d)$ is said to belong to $S^p_R$, $APW^p$ and $B^p$, respectively, if $f$ is a limit of a sequence of functions in $\rm{Trig}(\mathbb{R}^d)$ with respect to the norm $\|\cdot\|_{S^p_R}$, the semi-norms $\|\cdot\|_{W^p}$ and $\|\cdot\|_{B^p}$, respectively. Under the equivalent relation that $f\sim g$ if $\|f-g\|_{W^p}=0$ (resp., $\|f-g\|_{B^p}=0$), the set $APW^p/\sim$ (resp., $B^p/\sim$) becomes a Banach space.
Functions in $APW^2$ (resp., $B^2$) are said to be almost-periodic in the sense of H. Weyl (resp., Besicovich).

Note that if $0<r<R<\infty$, then
\begin{equation}\label{hom_inequ_1}\|f\|_{S_R^p}\leq C\|f\|_{S_r^p},\end{equation}
where $C$ depends only on $d$ and $p$. This implies that $\|f\|_{W^p}\leq C_p\|f\|_{S^p_R}$ for any $R>0$, and thereby,
$S^p_R\subset APW^p\subset B^p$.

Let $$L^p_{\rm{loc, unif}}(\mathbb{R}^d):=\Big\{f\in L^p_{loc}(\mathbb{R}^d): \sup_{x\in \mathbb{R}^d}\int_{B(x, 1)}|f|^p<\infty\Big\}.$$ It is not hard to see that $APW^p\subset L^p_{\textrm{loc, unif}}(\mathbb{R}^d)$. Moreover, for a function $f\in L^p_{\rm{loc, unif}}(\mathbb{R}^d)$, $f\in APW^p$ if and only if
\begin{equation}\label{hom_char}\sup_{y\in\mathbb{R}^d}\inf_{|z|\leq L}\|\Delta_{yz}(f)\|_{S^p_R}\rightarrow 0~~~\text{as}~L, R\rightarrow \infty,\end{equation}
where $\Delta_{yz}f(x):=f(x+y)-f(x+z)$ is the difference operator for $y, z\in \mathbb{R}^d$ (see \cite{Besicovitch1955, Shen2018_Approximate}).

Let $f\in L^1_{loc}(\mathbb{R}^d)$. A number $\langle f\rangle$ is called the mean value of $f$ if
\begin{equation*}
\lim\limits_{\varepsilon\rightarrow 0^+}\int_{\mathbb{R}^d}f(x/\varepsilon)\varphi(x) dx=\langle f\rangle\int_{\mathbb{R}^d}\varphi
\end{equation*}
for any $\varphi\in C^\infty_c(\mathbb{R}^d)$. If $f\in L^2_{loc}(\mathbb{R}^d)$ and $\|f\|_{B^2}<\infty$, the existence of $\langle f\rangle$ is equivalent to the condition that $f(x/\varepsilon)\rightharpoonup\langle f\rangle$ weakly in $L^2_{loc}(\mathbb{R}^d)$ as $\varepsilon\rightarrow 0$. In this case, one has
$$\langle f\rangle=\lim\limits_{L\rightarrow \infty}\fint_{B(0, L)}f. $$
If $f, g\in B^2$, then $fg$ has the mean value $\langle fg\rangle$ and the space $B^2$ is a Hilbert space endowed with the inner product $(f, g):=\langle fg\rangle$. Moreover, if $f\in B^2, g\in L^\infty(\mathbb{R}^d)\cap B^2$, then $fg\in B^2$.

Denote by $\overline{m}$ the number of the multi-indexes $\alpha$ satisfying $|\alpha|=m$. Let $\mathcal{V}$ and $\mathcal{W}$ be the closures of $\{\{D^\alpha\varphi\}_{|\alpha|=m}: \varphi\in \rm{Trig}(\mathbb{R}^d), \langle\varphi\rangle=0\}$ and $\{\{\varphi_\alpha\}_{|\alpha|=m}: \varphi\in \rm{Trig}(\mathbb{R}^d; \mathbb{R}^{\overline{m}}), \sum_\alpha D^\alpha\varphi_\alpha=0, \langle\varphi\rangle=0\}$ in $(B^2)^{\overline{m}}$ respectively. Then $(B^2)^{\overline{m}}$ has the decomposition (see e.g. \cite{Jikov1994})
 $$(B^2)^{\overline{m}}=\mathcal{V}\oplus\mathcal{W}\oplus\mathbb{R}^{\overline{m}}.$$

\subsection{Qualitative homogenization and compactness results}\label{sec_hom}

Throughout this subsection, we always assume that $A=(A^{\alpha\beta}_{ij})$ satisfies \eqref{intro_cond_bdd}--\eqref{intro_cond_el} and $A\in B^2$, i.e., $A_{ij}^{\alpha\beta}\in B^2$. It follows from \eqref{intro_cond_el} that
\begin{align*}\langle A_{ij}^{\alpha\beta}\phi_i^\alpha\phi_j^\beta\rangle\geq \mu\langle\phi^2\rangle ~\textrm{ for any }~ \phi\in\mathcal{V}^n,
\end{align*}
which, by the Lax-Milgram theorem, yields that, for any $\beta$ with $|\beta|=m$ and $1\leq j\leq n$, there exists a unique $\psi_j^\beta=(\psi_{ij}^{\alpha\beta})\in \mathcal{V}^n$ such that
\begin{equation}\label{hom_eq_corrector}
\langle A_{ik}^{\alpha\gamma}\psi_{kj}^{\gamma\beta}\phi_i^\alpha\rangle=-\langle A_{ij}^{\alpha\beta}\phi_i^\alpha\rangle~~\text{  for any } \phi=(\phi_i^\alpha)\in \mathcal{V}^n.
\end{equation}
Denote \begin{equation}\label{hom_def_A}\widehat{A}_{ij}^{\alpha\beta}=\langle A_{ij}^{\alpha\beta}\rangle+\langle A_{ik}^{\alpha\gamma}\psi_{kj}^{\gamma\beta}\rangle.\end{equation}
Then one can show that
$
|\widehat{A}^{\alpha \beta}_{ij}|\leq \mu_1,
$
where $\mu_1$ depends only on $\mu$, and
\begin{align*}
\sum_{|\alpha|=|\beta|=m}\int_{\mathbb{R}^d} \widehat{A}_{ij}^{\alpha\beta} D^\beta\phi_j D^\alpha\phi_i \geq \mu \|D^\alpha\phi\|^2_{L^2(\mathbb{R}^d)} ~\textrm{ for any}~\phi\in C_c^\infty(\mathbb{R}^d; \mathbb{R}^n).
\end{align*}
Moreover, by replacing $A$ in \eqref{hom_eq_corrector} by its adjoint $A^*$, we can get a unique solution $\psi^*\in\mathcal{V}^n$, such  that $\widehat{A^*}=(\widehat{A})^*$, where $\widehat{A^*}$ is defined as \eqref{hom_def_A} with $A$, $\psi$ replaced by $A^*$, $\psi^*$ respectively.

Next we provide two compactness theorems, which will be used to establish the large-scale H\"{o}lder estimates as well as a Liouville theorem for elliptic systems
\begin{equation*}
\mathcal{L}_\varepsilon u_\varepsilon+\lambda u_\varepsilon =\sum_{|\alpha|\leq m}D^\alpha f_\alpha,
\end{equation*}
with $B^2$-coefficients. We need the following lemma (\cite[Lemma 1.1]{Jikov1994}).

\begin{lemma}\label{hom_lem_div_curl}
Let $\{u_l\}, \{v_l\}$ be two bounded sequences in $L^2(\Omega; \mathbb{R}^{\overline{m}})$. Suppose that
\begin{enumerate}
      \item [(1)]$u_l\rightharpoonup u$ and $v_l\rightharpoonup v$ weakly in $L^2(\Omega; \mathbb{R}^{\overline{m}})$;
      \item [(2)]$u_l=(D^\alpha U_l)$ for some $U_l\in L^2(\Omega)$ and $D^\alpha(v_l)_\alpha\rightarrow f$ in $H^{-m}(\Omega)$.
\end{enumerate}
Then for any scalar function $\varphi\in C^\infty_c(\Omega)$,$$\int_\Omega (u_l\cdot v_l)\varphi dx\rightarrow\int_\Omega (u\cdot v)\varphi dx ~\textrm{ as } l\rightarrow\infty.$$
\end{lemma}

\begin{theorem}\label{hom_thm_compactness}
Let $u_l\in H^m(\Omega; \mathbb{R}^n)$ be a weak solution to $\mathcal{L}^{A_l}_{\varepsilon_l}(u_l)+\lambda_l u_l=f_l$ in $\Omega$, where $\varepsilon_l \rightarrow 0$, $\lambda_l\geq 0$, $\lambda_l\rightarrow \lambda$ and $A_l(x)=A(x+x_l)$ for some $x_l\in \mathbb{R}^d$. Assume that $u_l\rightharpoonup u$ weakly in $H^m(\Omega; \mathbb{R}^n)$, $f_l\rightarrow f$ strongly in $H^{-m}(\Omega; \mathbb{R}^n)$ and $\{x_l\}$ is bounded in $\mathbb{R}^d$. Then $A_l^{\alpha\beta}(x/\varepsilon_l)D^\beta u_l\rightharpoonup\widehat{A}^{\alpha\beta}D^\beta u$ weakly in $L^2(\Omega; \mathbb{R}^{\overline{m}\times n})$ and, consequently, $u$ is a weak solution to $\mathcal{L}_0(u)+\lambda u=f$ in $\Omega$.
\end{theorem}

\begin{proof}Since $\{p_l\}:=\{(A^{\alpha\beta}_{l, ij}(x/\varepsilon_l)D^\beta u_{l, j})\}$ is uniformly bounded in $L^2(\Omega; \mathbb{R}^{\overline{m}\times n})$, it is sufficient to show $p_0=\widehat{A}^{\alpha\beta}D^\beta u$, assuming that $\{p_l\}$ is weakly convergent to $p_0$ in $L^2(\Omega; \mathbb{R}^{\overline{m}\times n})$.

Fix $1\leq k \leq n$ and $\gamma$ with $|\gamma|=m$ and denote $q_j^\beta=A^{\alpha\beta}_{ij}(\psi^{*\alpha\gamma}_{ik}+\delta_{ik}\delta^{\alpha\gamma})$. Consider the identity
\begin{equation}\label{hom_iden}\int_\Omega p_{l, i}^\alpha\cdot(\psi^{*\alpha\gamma}_{l, ik}(x/\varepsilon_l)+\delta_{ik}\delta^{\alpha\gamma})\varphi dx=\int_\Omega D^\beta u_{l, j}\cdot q_{l, j}^\beta(x/\varepsilon_l)\cdot\varphi dx,\end{equation}
where $\psi^{*}_{l}(x)=\psi^{*}(x+x_l)$, $q_{l, j}^\beta(x)=q_j^\beta(x+x_l)$, $\varphi\in C^\infty_c(\Omega)$ is arbitrary and $\delta$ is the Kronecker delta function. Since $\psi^{*\gamma}_k\in \mathcal{V}^n$, there exists a sequence $\{\Psi^\tau\}\subset\rm{Trig}(\mathbb{R}^d)$ such that $\langle\Psi^\tau\rangle=0$ and $D^\alpha\Psi^\tau\rightarrow \psi^{*\alpha\gamma}_k$ in $B^2$ as $\tau\rightarrow\infty$. By setting $\Psi_l^\tau(x)$=$\Psi^\tau(x+x_l)$, we have
\begin{align*}
\int_\Omega p_{l, i}^\alpha\cdot(\psi^{*\alpha\gamma}_{l, ik}(x/\varepsilon_l)+\delta_{ik}\delta^{\alpha\gamma})\varphi dx&=\int_\Omega p_{l, i}^\alpha\cdot(D^\alpha\Psi_{l, i}^\tau(x/\varepsilon_l)+\delta_{ik}\delta^{\alpha\gamma})\varphi dx\\&\quad+\int_\Omega p_{l, i}^\alpha\cdot(\psi^{*\alpha\gamma}_{l, ik}(x/\varepsilon_l)-D^\alpha\Psi_{l, i}^\tau(x/\varepsilon_l))\varphi dx\\
&\doteq I_{1, l}+I_{2, l}.
\end{align*}
Note that $D^\alpha p_{l}^\alpha=(-1)^m(f_l-\lambda_l u_l)\rightarrow (-1)^m(f-\lambda u)$ strongly in $H^{-m}(\Omega)$ and $D^\alpha\Psi_l^\tau(x/\varepsilon_l)\rightharpoonup0$ weakly in $L^2(\Omega; \mathbb{R}^{\overline{m}\times n})$ as $l\rightarrow \infty$. Applying Lemma \ref{hom_lem_div_curl} to $I_{1, l}$, we get
$$\lim\limits_{l\rightarrow\infty} I_{1, l}=\int_\Omega p_{0, k}^\gamma\cdot\varphi dx.$$
For the second part, we have
\begin{align*}
|I_{2, l}|&\leq C\sum_\alpha\Big(\int_\Omega |\psi^{*\alpha\gamma}_{l, k}(x/\varepsilon_l)-D^\alpha\Psi_l^\tau(x/\varepsilon_l)|^2 dx\Big)^{1/2}\\
&\leq C\sum_\alpha\Big(\int_{\varepsilon_lx_l+\Omega} |\psi^{*\alpha\gamma}_{k}(x/\varepsilon_l)-D^\alpha\Psi^\tau(x/\varepsilon_l)|^2 dx\Big)^{1/2}.
\end{align*}
Since $\{x_l\}$ is bounded, we can find some $R>0$ such that $\varepsilon_lx_l+\Omega \subset B(0, R)$ for each $l$, where $R$ depends only on $\Omega$ and the bound of $\{|x_l|\}$. As a result,
\begin{align*}
|I_{2, l}|&\leq C\sum_\alpha\Big(\fint_{B(0, R)} |\psi^{*\alpha\gamma}_{k}(x/\varepsilon_l)-D^\alpha\Psi^\tau(x/\varepsilon_l)|^2 dx\Big)^{1/2},
\end{align*}
which implies that $\lim\limits_{\tau\rightarrow\infty}\lim\limits_{l\rightarrow\infty}I_{2, l}\leq\lim\limits_{\tau\rightarrow\infty}C\|\psi^{*\gamma}_{k}-\nabla^m\Psi^\tau\|_{B^2}=0$. Therefore,
\begin{align}\label{hom_lim1}
\lim\limits_{l\rightarrow\infty}\int_\Omega p_{l, i}^\alpha\cdot(\psi^{*\alpha\gamma}_{ik}(x/\varepsilon_l)+\delta_{ik}\delta^{\alpha\gamma})\varphi dx=\int_\Omega p_{0 k}^\gamma\cdot\varphi dx.
\end{align}
For the r.h.s. of \eqref{hom_iden}, noticing that $q\in\mathcal{W}^n\oplus\mathbb{R}^{\overline{m}\times n}$ and applying a similar argument as above, we can obtain
\begin{align}\label{hom_lim2}
\lim\limits_{l\rightarrow\infty}\int_\Omega D^\beta u_{l, j}\cdot q_{l, j}^\beta(x/\varepsilon_l)\cdot\varphi dx&=\int_\Omega D^\beta u_j\cdot \widehat{A}^{\gamma\beta}_{kj}\cdot\varphi dx.
\end{align}
By combining \eqref{hom_iden}, \eqref{hom_lim1} and \eqref{hom_lim2}, we conclude that $$\int_\Omega p_{0, k}^\gamma\cdot\varphi dx=\int_\Omega D^\beta u_j\cdot \widehat{A}^{\gamma\beta}_{kj}\cdot\varphi dx,$$which yields $p_{0, k}^\gamma=\widehat{A}^{\gamma\beta}_{kj}D^\beta u_j$. The proof is completed.
\end{proof}
As a corollary, we know that the homogenized operator of $\mathcal{L}_\varepsilon$ is $\mathcal{L}_0:=(-1)^mD^\alpha(\widehat{A}^{\alpha\beta}D^\beta)$.
\begin{coro}\label{hom_coro_hom}
Let $u_\varepsilon\in H^m(\Omega; \mathbb{R}^n)$ be the weak solution to Dirichlet problem \eqref{intro_eq1} with $f\in H^{-m}(\Omega; \mathbb{R}^n)$ and $\dot{g}\in W\!A^{m, 2}(\partial\Omega, \mathbb{R}^n)$. Then as $\varepsilon\rightarrow0$, $u_\varepsilon\rightarrow u_0$ weakly in $H^m(\Omega; \mathbb{R}^n)$, where $u_0$ is the solution to problem \eqref{intro_eq_hom}.
\end{coro}

\begin{theorem}\label{hom_thm_compactness2}
Let $u_l\in H^m(\Omega; \mathbb{R}^n)$ be a weak solution to $\mathcal{L}^{A_{y_l}}_{\varepsilon_l}(u_l)+\lambda_l u_l=f_l$ in $\Omega$, where $\varepsilon_l \rightarrow 0$, $\lambda_l\geq 0$, $\lambda_l\rightarrow \lambda$ and $A_{y_l}(x):=A(x+y_l)$ for some $y_l\in \mathbb{R}^d$. Assume that $u_l\rightharpoonup u$ weakly in $H^m(\Omega; \mathbb{R}^n)$ and $f_l\rightarrow f$ strongly in $H^{-m}(\Omega; \mathbb{R}^n)$. Suppose further
\begin{align}
\lim_{L\rightarrow\infty}\lim_{r\rightarrow \infty}\sup_{y\in\mathbb{R}^d}\inf_{|z|\leq L}\Big(\fint_{B(0, r)}|\Delta_{yz}A|^2\Big)^{1/2}=0. \label{hom_condition_Deltayz}
\end{align}
Then $u$ is a weak solution to $\mathcal{L}_0(u)+\lambda u=f$ in $\Omega$.
\end{theorem}

\begin{proof}
Let $R$ be a positive constant, depending only on $\Omega$, such that $\Omega\subset B(0, R)$. For each $l$, we can find $z_{l, L}\in\mathbb{R}^d$ with $|z_{l, L}|\leq L$ such that
\begin{align}
\Big(\fint_{\frac{\Omega}{\varepsilon_l}}|\Delta_{y_l z_{l, L}}A|^2\Big)^{1/2}\leq \sup_{y\in\mathbb{R}^d}\inf_{|z|\leq L}\Big(\fint_{\frac{\Omega}{\varepsilon_l}}|\Delta_{yz}A|^2\Big)^{1/2}\leq C\sup_{y\in\mathbb{R}^d}\inf_{|z|\leq L}\Big(\fint_{B(0, \frac{R}{\varepsilon_l})}|\Delta_{yz}A|^2\Big)^{1/2},\label{hom_ineq_compact_yz}
\end{align}
where $\Omega/\varepsilon_l=\{x\in \mathbb{R}^d: \varepsilon_l x\in\Omega\}$ and $C$ depends only on $\Omega$.
Consider the auxiliary system
\begin{equation*}
\begin{cases}
\mathcal{L}_{\varepsilon_l}^{A_{z_{l, L}}}v_{l, L}+\lambda_l v_{l, L}=f_l & \textrm{ in } \Omega,\\
Tr(D^\gamma v_{l, L})=D^\gamma u_l  & \text{ on } \partial \Omega, \text{ for  } 0\leq|\gamma|\leq m-1,
\end{cases}
\end{equation*}
and set $w_{l, L}=v_{l, L}-u_l$. Then $w_{l, L}\in H_0^m(\Omega; \mathbb{R}^n)$ satisfies
\begin{align}\label{hom_eq_auxiliary2}
\mathcal{L}^{A_{z_{l, L}}}_{\varepsilon_l}w_{l, L}+\lambda_l w_{l, L}=(-1)^m\sum_{|\alpha|=|\beta|=m}D^{\alpha}[(A^{\alpha\beta}_{y_l}-A^{\alpha\beta}_{z_{l, L}})(\frac{x}{\varepsilon_l})D^\beta u_l]~\textrm{ in } \Omega.
\end{align}
For $F\in L^2(\Omega; \mathbb{R}^n)$, let $\widetilde{u}_{l, L}\in H^m_0(\Omega; \mathbb{R}^n)$ be the weak solution to
\begin{equation*}
\begin{cases}
\mathcal{L}_{\varepsilon_l}^{A^*_{z_{l, L}}}\widetilde{u}_{l, L}+\lambda_l \widetilde{u}_{l, L}=F & \textrm{ in } \Omega,\\
Tr(D^\gamma \widetilde{u}_{l, L})=0  & \text{ on } \partial \Omega, \text{ for  } 0\leq|\gamma|\leq m-1.
\end{cases}
\end{equation*}
Thanks to the $W^{m, p}$ estimate for higher-order elliptic systems in \cite{Dong2011_Higher}, there exists a constant $q>2$, depending only on $m, n, \Omega, \mu$, such that,
\begin{align}
\|\nabla^m \widetilde{u}_{l, L}\|_{L^q(\Omega)} \leq C[\|F\|_{L^2(\Omega)}+\lambda_l\|\widetilde{u}_{l, L}\|_{L^2(\Omega)}]
 \leq C\|F\|_{L^2(\Omega)},\label{hom_ineq_reverse}
\end{align}
where $C$ depends only on $m ,n, \Omega, \mu$ and the upper bound of $\{\lambda_l\}$.
By \eqref{hom_ineq_compact_yz}--\eqref{hom_ineq_reverse}, we deduce that
\begin{align*}
\langle w_{l, L}, F\rangle_{L^2(\Omega)\times L^2(\Omega)}&=\int_\Omega(A^{\alpha\beta}_{y_l}-A^{\alpha\beta}_{z_{l, L}})(\frac{x}{\varepsilon_l})D^\beta u_l D^\alpha \widetilde{u}_{l, L}\nonumber\\
&\leq C\Big(\fint_\Omega\Big|(A_{y_l}-A_{z_{l, L}})(\frac{x}{\varepsilon_l})\Big|^p dx\Big)^{1/p}\Big(\int_\Omega|\nabla^m \widetilde{u}_{l, L}|^q\Big)^{1/q}\nonumber\\
&\leq C\|F\|_{L^2(\Omega)}\sup_{y\in\mathbb{R}^d}\inf_{|z|\leq L}\Big(\fint_{B(0, R/\varepsilon_l)}|\Delta_{yz}A|^2\Big)^{1/p},
\end{align*}
where $p>2$ satisfies $1/p+1/q=1/2$. This, together with \eqref{hom_condition_Deltayz}, implies that
\begin{align}
\lim_{L\rightarrow\infty}\lim_{l\rightarrow \infty}\|w_{l, L}\|_{L^2(\Omega)}=0. \label{hom_es_wl}
\end{align}

Now for each $L$, we may assume that  $v_{l, L}\rightharpoonup v^L$ in $H^m(\Omega; \mathbb{R}^n)$ as $l\rightarrow \infty$. Since $|z_{l, L}|\leq L$, according to Theorem \ref{hom_thm_compactness}, $v^L$ is a weak solution of $\mathcal{L}_0(u)+\lambda u=f$ in $\Omega$. Note that $w_{l, L}=v_{l, L}-u_l\rightharpoonup v^L-u$ in $H^m(\Omega; \mathbb{R}^n)$ as $l\rightarrow \infty$, and by \eqref{hom_es_wl},
\begin{align*}
\lim_{L\rightarrow\infty}\|v^L-u\|_{L^2(\Omega)}=0.
\end{align*}
Consequently, $v^L\rightharpoonup u$ weakly in $H^m(\Omega; \mathbb{R}^n)$ and $u$ is a weak solution to $\mathcal{L}_0(u)+\lambda u=f$ in $\Omega$. This completes the proof.
\end{proof}

\begin{remark}
If $A\in APW^2(\mathbb{R}^d)$, then condition \eqref{hom_condition_Deltayz} holds.
\end{remark}

\section{Some technical lemmas}\label{sec_pre}

\begin{lemma}\label{pre_lem_poincare}
Let $\mathcal{P}_{m-1}$ be the space of polynomials of degree at most $(m-1)$. Then there exists a family of linear operators $\{P_{m-1}(\cdot; x_0, r): W^{m, \frac{2d}{d+2}}(B(x_0, r))\rightarrow\mathcal{P}_{m-1}\}$, $x_0\in \mathbb{R}^d$, $r>0$, satisfying the following properties:
\begin{enumerate}
\item[(i)] for any $P\in\mathcal{P}_{m-1}$, $P_{m-1}(P; x_0, r)=P$;
\item[(ii)] for $v(x)=u(rx+x_0)$, $P_{m-1}(v; 0, 1)(x)=P_{m-1}(u; x_0, r)(rx+x_0)$;
\item[(iii)] for any $u\in H^m(B(x_0 ,r))$, $\frac{2d}{d+2}\leq p\leq 2$, the coefficient of $(x-x_0)^\alpha$ in $P_{m-1}(u; x_0, r)$ is bounded by $Cr^{-d/p-|\alpha|}\|u\|_{L^p(B(x_0, r))}$, where $C$ depends only on $p$. Thus,  if $u_l\rightarrow u$ in $L^2(B(x_0, r))$, then the coefficients of $P_{m-1}(u_l; x_0, r)$ converges to those of $P_{m-1}(u; x_0, r)$;
\item[(iv)] for any $u\in H^m(B(x_0, r))$, $\frac{2d}{d+2}\leq p\leq 2$, \begin{equation}\label{pre_ineq_poincare} \|u-P_{m-1}(u; x_0, r)\|_{L^2(B(x_0, r))}\leq Cr^{m+\frac{d}{2}-\frac{d}{p}}\|\nabla^m u\|_{L^p(B(x_0, r))},\end{equation}where $C$ depends only on $p$ and $m$.
\end{enumerate}
\end{lemma}
\begin{proof}
According to Theorem 8.11 and 8.12 in \cite{Lieb2001}, for any bounded Lipschitz domain $\Omega$, by choosing $\{f_\alpha: |\alpha|\leq m-1\}\subset L^{2^*}(\Omega)$ with \begin{align}\int_\Omega f_\alpha(x)x^\beta dx=\delta_{\alpha\beta}, ~\textrm{for any}~ |\alpha|, |\beta|\leq m-1,\label{pre_iden_delta}\end{align}we have for $\frac{2d}{d+2}\leq p\leq 2$
\begin{equation}\label{pre_ineq_poin}
\bigg\|u-\sum_{|\alpha|\leq m-1}x^\alpha\int_\Omega uf_\alpha\bigg\|_{H^{m-1}(\Omega)}\leq C\|\nabla^m u\|_{L^p(\Omega)},
\end{equation}
where $C$ depends only on $\Omega, f_\alpha, p$ and $m$. Now fix $\{f_\alpha\}$ for $\Omega=B(0, 1)$ and define for $x_0\in\mathbb{R}^d, r>0,$ $$P_{m-1}(u; x_0, r):=\sum_{|\alpha|\leq m-1}r^{-d-|\alpha|}(x-x_0)^\alpha\int_{B(x_0, r)} u(y)f_{\alpha}(\frac{y-x_0}{r})dy.$$ Obviously, $P_{m-1}$ is linear, and, in view of \eqref{pre_iden_delta} and \eqref{pre_ineq_poin}, it is not hard to verify that $P_{m-1}(\cdot; x_0, r)$ satisfies properties (i)--(iv).
\end{proof}

\begin{remark}
The operator $P_{m-1}(\cdot; x_0, r)$, depending on the choice of $\{f_\alpha\}$, may not be unique. The coefficients of $P_{m-1}(u; x_0, r)$ depend only on the $L^p$-norm of $u$, but never  on the norms of the derivatives of $u$.
\end{remark}

\begin{lemma}\label{pre_lem_cacci}
Assume that $A$ satisfies \eqref{intro_cond_bdd}--\eqref{intro_cond_el}. Let $B=B(x_0, r)$, $2B=B(x_0, 2r)$ be balls in $\mathbb{R}^d$, and $u\in H^m(2B; \mathbb{R}^n )$ be a solution to $\mathcal{L}_1 u+\lambda u= \sum_{|\alpha|\leq m}D^\alpha f_\alpha$ in $2B$, where $\lambda\geq 0$ and $f_\alpha \in L^2(2B; \mathbb{R}^n )$ for $|\alpha|\leq m$. Then
 there exists a constant $C$, depending only on $d, m, n$ and $\mu$, such that for $0\leq j\leq m$,
\begin{align}
\int_B |\nabla^j u|^2\leq \frac{C}{r^{2j}} \int_{2B}|u|^2+C\lambda^2 r^{4m-2j}\int_{2B}|u|^2
+C\sum_{|\alpha|\leq m} r^{4m-2|\alpha|-2j} \int_{2B} |f_\alpha|^2.\label{pre_ineq_cacci1}
\end{align}
Moreover, if $\lambda>0$, we also have
\begin{align}
\sum_{k\leq m}\lambda^{\frac{m-k}{m}}\int_B |\nabla^k u|^2\leq   \frac{C}{r^{2m}} \int_{2B}|u|^2 +C\sum_{|\alpha|\leq m}\lambda^{\frac{|\alpha|-m}{m}} \int_{2B} |f_\alpha|^2. \label{pre_ineq_cacci2}
\end{align}
\end{lemma}
\begin{proof}
Estimate \eqref{pre_ineq_cacci1} can be proved in the same way as \cite[Corollary 22]{Barton2016_Gradient} by an induction argument (see \cite[Lemma 2.4]{Niu2019_Boundary}).
To prove \eqref{pre_ineq_cacci2}, by interpolation and rescaling, it is sufficient to show for $\lambda=1$,
\begin{align}\label{pre_ineq_cacci_lm}
\int_B |\nabla^m u|^2+\int_B |u|^2\leq   \frac{C}{r^{2m}} \int_{2B}|u|^2
+C\sum_{|\alpha|\leq m} \int_{2B} |f_\alpha|^2.
\end{align}
We will prove that
\begin{equation}\label{pre_ineq_cacci_claim}
\int_{B(x_0, \rho)}|\nabla^m u|^2+\int_{B(x_0, \rho)}|u|^2\leq \sum_{k<m}\frac{C}{(s-\rho)^{2m-2k}}\int_{B(x_0, s)\setminus B(x_0, \rho)}|\nabla^k u|^2+C\sum_{|\alpha|\leq m}\int_{2B}|f_\alpha|^2,
\end{equation}
whenever $0<\rho<s<2r$, which, together with Theorem 18 in \cite{Barton2016_Gradient}, gives \eqref{pre_ineq_cacci_lm}. Let $\varphi$ be a function in $C_c^\infty(B(x_0, s))$ such that $\varphi\equiv1$ in $B(x_0, \rho)$ and $|\nabla^k \varphi|\leq \frac{C}{|s-\rho|^k}$ in $B(x_0, s)$. Taking $u\varphi^{2m}$ as the test function, we obtain
\begin{gather}\int_{\mathbb{R}^d}|\nabla^m (u\varphi^{m})|^2+\int_{\mathbb{R}^d}|u\varphi^{m}|^2\leq C\sum_{k<m}(s-\rho)^{-2(m-k)}\int_{B(x_0, s)\setminus B(x_0, \rho)}|\nabla^k u|^2\nonumber\\+C_\delta\sum_{|\alpha|\leq m}\int_{2B}|f_\alpha|^2+\delta\sum_{|\alpha_1+\alpha_2|\leq m}(s-\rho)^{-2|\alpha_2|}\int_{B(x_0, s)\setminus B(x_0, \rho)}|D^{\alpha_1}(u\varphi^{m})|^2\varphi^{2m-2|\alpha_2|},\label{pre_ineq_cacci_test}\end{gather}
where $\delta$ is a small constant and $C_\delta$ depends on $\delta$. If $s-\rho\leq1$, then $(s-\rho)^{-2|\alpha_2|}\leq (s-\rho)^{-2(m-|\alpha_1|)}$ for $|\alpha_1+\alpha_2|\leq m$. This, together with \eqref{pre_ineq_cacci_test}, implies \eqref{pre_ineq_cacci_claim} for $\delta$ small. If $s-\rho\geq1$, then $(s-\rho)^{-2|\alpha_2|}\leq C$. Therefore, by interpolation,
\begin{align*}\sum_{|\alpha_1+\alpha_2|\leq m}(s-\rho)^{-2|\alpha_2|}\int_{B(x_0, s)\setminus B(x_0, \rho)}|D^{\alpha_1}(u\varphi^{m})|^2\varphi^{2m-2|\alpha_2|}\leq C\int_{\mathbb{R}^d}|\nabla^m (u\varphi^{m})|^2+C\int_{\mathbb{R}^d}|u\varphi^{m}|^2,
\end{align*}
which, combined with \eqref{pre_ineq_cacci_test}, implies \eqref{pre_ineq_cacci_claim} for $\delta$ small. The  proof is thus completed.
\end{proof}
\begin{remark}Since $u-P_{m-1}$ satisfies $\mathcal{L}_1 (u-P_{m-1})= \sum_{|\alpha|=m}D^\alpha f_\alpha-\lambda u$ for any $P_{m-1}\in \mathcal{P}_{m-1}$, it follows that
\begin{align}\label{pre_ineq_cacci_remark}
\int_B |\nabla^m u|^2\leq \frac{C}{r^{2m}} \int_{2B}|u-P_{m-1}|^2+C\lambda^2 r^{2m}\int_{2B}|u|^2
+C\sum_{|\alpha|\leq m} r^{2m-2|\alpha|} \int_{2B} |f_\alpha|^2.
\end{align}
\end{remark}

\begin{lemma}\label{pre_lem_reverse}
Assume that the assumptions of Lemma \ref{pre_lem_cacci} hold. Then there exists some $q^+>2$ depending only on $d, m, n$ and $\mu$, such that, for $2\leq q\leq q^+$,
\begin{enumerate}
\item[(i)] if $\lambda=0$,
\begin{align}
\Big(\fint_B |\nabla^m u|^q\Big)^{1/q}&\leq C\bigg\{\Big(\fint_{2B}|\nabla^m u|^2\Big)^{1/2} \nonumber\\&\quad\quad\,+\sum_{|\alpha|<m}r^{m-|\alpha|}\Big(\fint_{2B}|f_\alpha|^2\Big)^{1/2}+ \sum_{|\alpha|=m}\Big(\fint_{2B}|f_\alpha|^q\Big)^{1/q}\bigg\};\label{pre_ineq_reverse1}
\end{align}
\item[(ii)] if $\lambda>0$,
\begin{align}
&\quad\Big(\fint_B |\nabla^m u|^q\Big)^{1/q}+\sqrt{\lambda}\Big(\fint_B |u|^q\Big)^{1/q}\nonumber\\
& \leq C\bigg\{\Big(\fint_{2B}|\nabla^m u|^2\Big)^{1/2}+\sqrt{\lambda}\Big(\fint_{2B}|u|^2\Big)^{1/2}+\sum_{|\alpha|\leq m}\lambda^{\frac{|\alpha|-m}{2m}}\Big(\fint_{2B}|f_\alpha|^q\Big)^{1/q}\bigg\};\label{pre_ineq_reverse2}
\end{align}
\end{enumerate}
where $C$ depends only on $d, m, n$ and $\mu$.
\end{lemma}
\begin{proof}
Estimate \eqref{pre_ineq_reverse1} follows from \cite[Theorem 24]{Barton2016_Gradient} and $W^{m, p}$ estimate for higher-order elliptic systems (\cite{Dong2011_Higher}). To show \eqref{pre_ineq_reverse2}, by rescaling, we assume $\lambda=1$. Let $x_1\in \mathbb{R}^d$ and $\rho>0$ such that $B(x_1, 2\rho)\subset B(x_0, 2r)$. Let $P(x)=P_{m-1}(u; x_1, 2\rho)$ be given by Lemma \ref{pre_lem_poincare}. Applying \eqref{pre_ineq_cacci2} to $\mathcal{L}_1 (u-P)+(u-P)=\sum_{|\alpha|\leq m}f_\alpha-P  \textrm{ in } B(x_1, 2\rho), $ we obtain
\begin{align}
&\quad\int_{B(x_1, \rho)} |\nabla^m u|^2+\int_{B(x_1, \rho)} |u-P|^2\nonumber\\& \leq \frac{C}{\rho^{2m}} \int_{B(x_1, 2\rho)}|u-P|^2+C\sum_{|\alpha|\leq m} \int_{B(x_1, 2\rho)} |f_\alpha|^2+C\int_{B(x_1, 2\rho)}|P|^2.\label{pre_ineq_reverse_proof1}
\end{align}
By properties (iii), (iv) in Lemma \ref{pre_lem_poincare}, for $p=\frac{2d}{d+2}$, we have
\begin{gather}
\Big(\int_{B(x_1, 2\rho)}|P|^2\Big)^{1/2}\leq C\rho^{\frac{d}{2}-\frac{d}{p}}\|u\|_{L^p(B(x_1, 2\rho))},\label{pre_ineq_reverse_proof3}\\
\Big(\int_{B(x_1, 2\rho)}|u-P|^2\Big)^{1/2}\leq C\rho^{m+\frac{d}{2}-\frac{d}{p}}\|\nabla^m u\|_{L^p(B(x_1, 2\rho))},\label{pre_ineq_reverse_proof2}
\end{gather}
where $C$ depends only on $d$ and $m$.
Thus, it follows from \eqref{pre_ineq_reverse_proof1}, \eqref{pre_ineq_reverse_proof3} and \eqref{pre_ineq_reverse_proof2}  that
\begin{align}
&\quad\Big(\fint_{B(x_1, \rho)} |\nabla^m u|^2\Big)^{1/2}+\Big(\fint_{B(x_1, \rho)} |u|^2\Big)^{1/2}\nonumber\\
&\leq C\bigg\{\Big(\fint_{B(x_1, 2\rho)}|\nabla^mu|^p\Big)^{1/p}+\Big(\fint_{B(x_1, 2\rho)}|u|^p\Big)^{1/p}+\sum_{|\alpha|\leq m} \Big(\fint_{B(x_1, 2\rho)} |f_\alpha|^2\Big)^{1/2}\bigg\}.\label{pre_ineq_reverse_proof4}
\end{align}
By the standard self-improving argument, we obtain \eqref{pre_ineq_reverse2} from \eqref{pre_ineq_reverse_proof4} immediately.
\end{proof}


\section{Approximate correctors}\label{sec_appcor}

Now we introduce the approximate correctors $\chi_T$ and establish some elementary estimates.
\begin{prop}\label{appcor_prop_1}
Suppose that $A$ satisfies \eqref{intro_cond_bdd}--\eqref{intro_cond_el} and $f_\alpha\in L^2_{\rm{loc, unif}}(\mathbb{R}^d; \mathbb{R}^n)$ for $|\alpha|\leq m$. Then, for any $T>0$, there exists a unique function $u\in H^m_{\rm{loc}}(\mathbb{R}^d; \mathbb{R}^n)$ such that $\nabla^k u\in L^2_{\rm{loc, unif}}(\mathbb{R}^d)$ for $0\leq k\leq m$, and
\begin{equation}\label{appcor_eq1}(-1)^{m}\sum_{|\alpha|=|\beta|=m} D^\alpha(A^{\alpha\beta}D^\beta u)+T^{-2m}u=\sum_{|\alpha|\leq m} D^\alpha f_\alpha \quad\text{in}~\mathbb{R}^d. \end{equation}
Moreover, $u$ satisfies
\begin{equation}\label{appcor_es_u}\sum_{k\leq m}T^{-m+k}\|\nabla^k u\|_{S_T^2}\leq C\sum_{|\alpha|\leq m}T^{m-|\alpha|}\|f_\alpha\|_{S_T^2},\end{equation} where $C$ depends only on $d, m, n$ and $\mu$.
\end{prop}
\begin{proof}
By rescaling we may assume that $T=1$. Let $\varphi(x)=\phi(x)+(1-\phi(x))e^{|x|}$ and $\varphi_\lambda(x)=\varphi(\lambda x)$, where $\phi\in C_c^\infty(B(0, 2))$ with  $\phi\equiv1$ in $B(0, 1)$. Observe that for $|\alpha|\leq m$, $|D^\alpha \varphi(x)|\leq C\varphi(x)$ with $C$ depending only on $m$, which implies that $|D^\alpha \varphi_\lambda(x)|\leq C\lambda^{|\alpha|}\varphi_\lambda(x)$ for $|\alpha|\leq m$.
This, together with the inequality $$\sum_{k<m}\int_{\mathbb{R}^d}|\nabla^k u|^2\varphi_\lambda\leq C\int_{\mathbb{R}^d}(|\nabla^m u|^2\varphi_\lambda+|u|^2\varphi_\lambda),$$ gives that, for $f_\alpha\in L^2(\mathbb{R}^d)$ with compact support, there exists a constant $\lambda>0$, depending only on $d, m, n$ and $\mu$, such that the solution of \eqref{appcor_eq1} satisfies
\begin{equation}\label{appcor_es_phi}
\sum_{k\leq m}\int_{\mathbb{R}^d}|\nabla^k u|^2\varphi_\lambda\leq C\int_{\mathbb{R}^d}\sum_{|\alpha|\leq m}|f_\alpha|^2\varphi_\lambda.
\end{equation}
With \eqref{appcor_es_phi} in hand, the proof may be completed in the same way as that for second-order systems in  \cite[Section 4]{Pozhidaev1989} or \cite{Shen2015_Convergence}. We therefore omit the details here.
\end{proof}

\begin{lemma}\label{appcor_lem_SR}
Suppose that $A$ satisfies \eqref{intro_cond_bdd}--\eqref{intro_cond_el}. Let $u\in H^m_{\rm{loc}}(\mathbb{R}^d; \mathbb{R}^n)$ be the weak solution to \eqref{appcor_eq1} in $\mathbb{R}^d$ given by Proposition \ref{appcor_prop_1}. Then for any $R\geq T$, the following  estimates hold
\begin{gather}
\sum_{k\leq m}T^{-m+k}\|\nabla^k u\|_{S_R^2}\leq C\sum_{|\alpha|\leq m}T^{m-|\alpha|}\|f_\alpha\|_{S_R^2},\label{appcor_ineq_SR}\\
\|\nabla^m u\|_{S_R^q}\leq C\sum_{|\alpha|\leq m}T^{m-|\alpha|}\|f_\alpha\|_{S_R^q},\label{appcor_es_u_reverse}\end{gather}
for $2\leq q\leq q^+$, where $q^+>2$ is given in Lemma \ref{pre_lem_reverse} and $C$ depends only on $d, m, n$ and $\mu$.
\end{lemma}

\begin{proof}
Estimate \eqref{appcor_ineq_SR} follows from Caccioppoli's inequality \eqref{pre_ineq_cacci2} and \eqref{hom_inequ_1}, while the estimate  \eqref{appcor_es_u_reverse} follows from the reverse H\"{o}lder inequality \eqref{pre_ineq_reverse2} and \eqref{appcor_ineq_SR}.
\end{proof}

For $T>0$, let $\chi_{T, l}^\gamma=(\chi_{T, jl}^\gamma)$ be the weak solution to
\begin{equation}\label{appcor_eq2}
(-1)^{m}\sum_{|\alpha|=|\beta|=m} D^\alpha(A^{\alpha\beta}D^\beta u)+T^{-2m}u=(-1)^{m+1}\sum_{|\alpha|=|\beta|=m} D^\alpha(A^{\alpha\beta}D^\beta P^\gamma_l)\quad\text{ in}~\mathbb{R}^d,
\end{equation}
given in Proposition \ref{appcor_prop_1}, where $P_l^\gamma=\frac{1}{\gamma!}x^\gamma e_l$, $|\gamma|=m$, $1\leq l\leq n$, and $e_l=(0, \dots, 1, \dots, 0)$ with $1$ in the $l$-th position. The matrix-valued functions $\chi_T=(\chi_{T, l}^\gamma)$ are called the approximate correctors. It follows from \eqref{appcor_ineq_SR} and \eqref{appcor_es_u_reverse} that, for $R\geq T$,
\begin{align}
\sum_{k\leq m}T^{-m+k}\|\nabla^k \chi_T\|_{S_R^2}+\|\nabla^m \chi_T\|_{S_R^{q^+}}\leq C,\label{appcor_es_chi_SR}
\end{align}
where $q^+>2$ and $C$ depends only on $d, m, n$ and $\mu$.

\begin{theorem}\label{appcor_thm_APW}
Suppose $A$ satisfies \eqref{intro_cond_bdd}--\eqref{intro_cond_el}. Then there exists some $2<p<\infty$, depending only on $d, m, n$ and $\mu$, such that  for any $y, z\in\mathbb{R}^d$ and $R\geq T$,
\begin{equation}\label{appcor_es_delta_chi}\sum_{k\leq m}T^{-m+k}\|\Delta_{yz}(\nabla^k \chi_T)\|_{S_R^2}\leq C\|\Delta_{yz}(A)\|_{S_R^p},\end{equation}
where $C$ depends only on $d, m, n$ and $\mu$. Furthermore, if $A\in APW^2$, then $\nabla^k \chi_T\in APW^2$ for each $0\leq k\leq m$.
\end{theorem}
\begin{proof}
Fix $1\leq l\leq n, |\gamma|=m$ and $y, z\in \mathbb{R}^d$. Let $u(x)=\chi^\gamma_{T, l}(x+y)-\chi_{T, l}^\gamma(x+z)$ and $v(x)=\chi_{T, l}^\gamma(x+z)$. Then
\begin{align*}
&\quad\ (-1)^{m}\sum_{|\alpha|=|\beta|=m} D^\alpha(A^{\alpha\beta}(x+y)D^\beta u)+T^{-2m}u\\
&=(-1)^{m+1}\sum_{|\alpha|=|\beta|=m} D^\alpha[(A^{\alpha\beta}(x+y)-A^{\alpha\beta}(x+z))D^\beta (v+P^\gamma_l)],
\end{align*}
  which, together with \eqref{appcor_ineq_SR} and \eqref{appcor_es_chi_SR}, gives \eqref{appcor_es_delta_chi} with $\frac{1}{p}+\frac{1}{q^+}=\frac{1}{2}$.
Furthermore, if $A\in APW^2$,  \eqref{hom_char} and \eqref{appcor_es_delta_chi} imply that $\nabla^k \chi_T \in  APW^2$ for each $0\leq  k \leq  m$. See  \cite[Lemma 3.2]{Shen2018_Approximate} for more details.
\end{proof}

It follows from  \eqref{appcor_eq2} and Theorem \ref{appcor_thm_APW} that if $A\in APW^2$ and $u=\chi_{T, l}^\gamma$, $1\leq l\leq n, |\gamma|=m$, then
$$\Big\langle\sum_{|\alpha|=|\beta|=m}A^{\alpha\beta}_{ij}D^\beta u_j D^\alpha v_i\Big\rangle+T^{-2m}\langle uv\rangle=-\Big\langle\sum_{|\alpha|=|\beta|=m}A^{\alpha\gamma}_{il}D^\alpha v_i\Big\rangle,$$
for any $v=(v_i)\in H^m_{\rm{loc}}(\mathbb{R}^d; \mathbb{R}^n)$ such that $\nabla^k v_i\in B^2$ for every $0\leq k\leq m$. This implies that
\begin{align}D^\alpha \chi_{T, ij}^\beta\rightarrow \psi_{ij}^{\alpha\beta} \quad\textrm{strongly in } B^2,~\mathrm{and}~T^{-2m}\langle |\chi_T|^2\rangle\rightarrow 0,\label{appcor_conver}\end{align}
as $T\rightarrow \infty$, where $\psi=(\psi_{ij}^{\alpha\beta})$ is defined by \eqref{hom_eq_corrector}. Moreover, by letting $v$ be a vector of constants, we get $\langle \chi_T\rangle=0$ for any $T>0$.


\section{H\"{o}lder estimates at large scale}\label{sec_holder}

In this section we establish the large-scale H\"{o}lder estimate for the approximate correctors $\chi_T$. As a byproduct, a Liouville theorem for  higher-order elliptic systems with $B^2$-coefficients is obtained. Throughout this section, unless indicated, we always assume that $A\in B^2$ satisfies \eqref{intro_cond_bdd}--\eqref{intro_cond_el} and condition \eqref{hom_condition_Deltayz}. Note that if  $A\in APW^2$, $A$ satisfies condition \eqref{hom_condition_Deltayz}.
\begin{theorem}\label{holder_thm_holder}
Fix $\sigma\in(0, 1)$ and $x_0\in\mathbb{R}^d$. Let $u_\varepsilon\in H^m(B(x_0, R); \mathbb{R}^n)$ be a weak solution to
\begin{equation*}\label{holder_eq_L1}
\mathcal{L}_\varepsilon u_\varepsilon+\lambda u_\varepsilon =\sum_{|\alpha|\leq m}D^\alpha f_\alpha \quad\textrm{ in } B(x_0, R),
\end{equation*} for $0<\varepsilon<R$ and $\lambda\in[0, R^{-2m}]$. Then  if $\varepsilon\leq r\leq R/2$,
\begin{align}
&\quad\Big(\fint_{B(x_0, r)}|\nabla^{m} u_\varepsilon|^2\Big)^{1/2}+\sqrt{\lambda}\Big(\fint_{B(x_0, r)}|u_\varepsilon|^2\Big)^{1/2}\nonumber\\
&\leq C_\sigma\Big(\frac{R}{r}\Big)^\sigma\bigg\{R^{-m}\Big(\fint_{B(x_0, R)}|u_\varepsilon|^2\Big)^{1/2}+\sum_{|\alpha|\leq m}\sup_{r\leq t\leq R}t^{m-|\alpha|}\Big(\fint_{B(x_0, t)}|f_\alpha|^2\Big)^{1/2}\bigg\},\label{holder_es_large}
\end{align}
where $C_\sigma$ depends only on $d, m ,n, \sigma$ and $A$.
\end{theorem}

We remark that estimate \eqref{holder_es_large} may fail for $0<r<\varepsilon$, since no smoothness condition on the coefficients is required. Also, one cannot expect further estimates, like Lipschitz estimates, on $u_\varepsilon$, since no additional condition on $\{f_\alpha\}$ is imposed.

In the following, we denote $P^r(u):=P_{m-1}(u; 0, r)$, where $P_{m-1}(\cdot; 0, r)$ is given by Lemma \ref{pre_lem_poincare}. To ensure our estimates are translation invariant when applying the compactness argument, we introduce the set of translations of $A$ by
\begin{align*}
\mathcal{A}=\{A_y: A_y(x)=A(x+y) \textrm{ for some } y\in \mathbb{R}^d\}.
\end{align*}

\begin{lemma}\label{holder_lem_onestep}
Fix $\sigma\in (0, 1)$. There exist $\varepsilon_0\in(0, 1/2)$ and $\theta\in(0,1/8)$, depending on $d, m, n, \sigma$ and $A$, such that
\begin{align}\label{holder_es_onestep}
&\Big(\fint_{B(0, \theta)}|u_\varepsilon-P^\theta(u_\varepsilon)|^2\Big)^{1/2}+\theta^m\sqrt{\lambda}\Big(\fint_{B(0, \theta)}|u_\varepsilon|^2\Big)^{1/2}\leq \theta^{m-\sigma}\bigg\{\Big(\fint_{B(0, 1)}|u_\varepsilon-P^1(u_\varepsilon)|^2\Big)^{1/2}\nonumber\\
&\qquad\qquad+\sqrt{\lambda}\Big(\fint_{B(0, 1)}|u_\varepsilon|^2\Big)^{1/2}+\varepsilon_0^{-1}\sum_{|\alpha|\leq m}\Big(\fint_{B(0, 1)}|f_\alpha|^2\Big)^{1/2}\bigg\},
\end{align}
whenever $0<\varepsilon<\varepsilon_0$, $\lambda\in [0, \varepsilon_0^2]$ and $u_\varepsilon\in H^m(B(0, 1); \mathbb{R}^n)$ is a weak solution to
\begin{equation}\label{holder_eq_L2}
\mathcal{L}^{\widetilde{A}}_\varepsilon u_\varepsilon+\lambda u_\varepsilon =\sum_{|\alpha|\leq m}D^\alpha f_\alpha \quad\text{ in } B(0, 1),
\end{equation}
for some $\widetilde{A}\in \mathcal{A}$.
\end{lemma}
\begin{proof}
\textbf{Claim:} Under the conditions of Lemma \ref{holder_lem_onestep}, there exist $\varepsilon_0\in(0, 1/2)$ and $\theta\in(0,1/8)$, depending on $\sigma$ and $A$, such that
\begin{align}\label{holder_es_claim}
&\quad\Big(\fint_{B(0, \theta)}|u_\varepsilon-P^\theta(u_\varepsilon)|^2\Big)^{1/2}+\theta^m\Big(\fint_{B(0, \theta)}|u_\varepsilon|^2\Big)^{1/2}\nonumber\\
&\leq \frac{\theta^{m-\sigma}}{2}\bigg\{\Big(\fint_{B(0, 1)}|u_\varepsilon|^2\Big)^{1/2}+\varepsilon_0^{-1}\sum_{|\alpha|\leq m}\Big(\fint_{B(0, 1)}|f_\alpha|^2\Big)^{1/2}\bigg\}.
\end{align}
We first show that \eqref{holder_es_claim} implies \eqref{holder_es_onestep}. In fact, since $u_\varepsilon-P^1(u_\varepsilon)$ is a weak solution to
\begin{equation*}
\mathcal{L}_\varepsilon u+\lambda u=\sum_{|\alpha|\leq m}D^\alpha f_\alpha-\lambda P^1(u_\varepsilon) \quad\text{ in } B(0, 1),
\end{equation*}
according to \eqref{holder_es_claim}, it holds that
\begin{align}
&\quad~\Big(\fint_{B(0, \theta)}|u_\varepsilon-P^\theta(u_\varepsilon)|^2\Big)^{\frac{1}{2}}\nonumber\\&=\Big(\fint_{B(0, \theta)}|(u_\varepsilon-P^1(u_\varepsilon))-P^\theta(u_\varepsilon-P^1(u_\varepsilon))|^2\Big)^{\frac{1}{2}}\label{holder_estimate_claim1}\\&\leq \frac{\theta^{m-\sigma}}{2}\bigg\{2\Big(\fint_{B(0, 1)}|u_\varepsilon-P^1(u_\varepsilon)|^2\Big)^{\frac{1}{2}}+\varepsilon_0^{-1}\sum_{|\alpha|\leq m}\Big(\fint_{B(0, 1)}|f_\alpha|^2\Big)^{\frac{1}{2}}+\sqrt{\lambda} \Big(\fint_{B(0, 1)}|u_\varepsilon|^2\Big)^{\frac{1}{2}}\bigg\},\nonumber
\end{align}
where we have used the linearity of $P_{m-1}$ and the property (i) of Lemma \ref{pre_lem_poincare} in the first step as well as the fact that $\lambda\leq \varepsilon_0^2$ in the last step. Moreover,  since $\lambda\in[0, 1)$, \eqref{holder_es_claim} implies that
\begin{align}\label{holder_estimate_claim2}
\theta^m\sqrt{\lambda}\Big(\fint_{B(0, \theta)}|u_\varepsilon|^2\Big)^{1/2}\leq \frac{\theta^{m-\sigma}}{2}\bigg\{\sqrt{\lambda}\Big(\fint_{B(0, 1)}|u_\varepsilon|^2\Big)^{1/2}+\varepsilon_0^{-1}\sum_{|\alpha|\leq m}\Big(\fint_{B(0, 1)}|f_\alpha|^2\Big)^{1/2}\bigg\}.
\end{align}
Combining \eqref{holder_estimate_claim1} and \eqref{holder_estimate_claim2}, we obtain \eqref{holder_es_onestep} immediately.

It remains to prove Claim \eqref{holder_es_claim}. If $u\in H^m(B(0, 1/2); \mathbb{R}^n)$ is a weak solution to
\begin{equation}\label{holder_eq_L0u}
\mathcal{L}_0u=0 \quad\text{ in } B(0, 1/2),
\end{equation}
by \eqref{pre_ineq_poincare} and the interior Lipschitz estimate (\cite{Niu2018_Convergence}), we have that
for any $\theta\in(0, 1/8)$,
\begin{equation}\label{holder_es_Lip}
\Big(\fint_{B(0, \theta)}|u-P^\theta(u)|^2\Big)^{1/2}+\theta^m\Big(\fint_{B(0, \theta)}|u|^2\Big)^{1/2}\leq C_0\theta^m\Big(\fint_{B(0, 1/2)}|u|^2\Big)^{1/2},
\end{equation}
where $C_0$ depends only on $d, m, n$ and $\mu$.
Now choose $\theta\in(0, 1/8)$ small enough such that $2^{d/2}C_0\theta^m<\frac{\theta^{m-\sigma}}{2}$. Suppose that \eqref{holder_es_claim} does not hold for this $\theta$ and any $\varepsilon_0\in(0, 1/2)$. Then there exist $A_l\subset \mathcal{A}$, $\{\varepsilon_l\}\subset\mathbb{R}_+$, $\{\lambda_l\}\subset[0, 1]$, $\{f_{\alpha, l}\}\subset L^2(B(0, 1); \mathbb{R}^n)$ for each $|\alpha|\leq m$ and $\{u_l\}\subset H^m(B(0, 1); \mathbb{R}^n)$, such that $\varepsilon_l\rightarrow 0$, $0\leq \lambda_l\leq \varepsilon_l^2$,
$$\mathcal{L}^{A_l}_{\varepsilon_l}(u_l)+\lambda_l u_l=\sum_{|\alpha|\leq m}D^\alpha f_{\alpha, l} \quad\text{ in } B(0, 1), $$ and moreover,
\begin{gather}
\Big(\fint_{B(0, 1)}|u_l|^2\Big)^{1/2}+\varepsilon_l^{-1}\sum_{|\alpha|\leq m}\Big(\fint_{B(0, 1)}|f_{\alpha, l}|^2\Big)^{1/2}\leq 1,\label{holder_ineq_contradict1}\\
\Big(\fint_{B(0, \theta)}|u_l-P^\theta(u_l)|^2\Big)^{1/2}+\theta^m\Big(\fint_{B(0, \theta)}|u_l|^2\Big)^{1/2}>\frac{\theta^{m-\sigma}}{2}. \label{holder_ineq_contradict2}\end{gather}
Thanks to Caccioppoli's inequality \eqref{pre_ineq_cacci1}, $\{u_l\}$ is bounded in $H^m(B(0, 1/2))$. By passing to a subsequence, we may assume that $u_l\rightarrow u$ weakly in $H^m(B(0, 1/2))$ and  $L^2(B(0, 1))$, and  strongly in $H^{m-1}(B(0, 1/2))$. Furthermore, note that $\lambda_l\rightarrow 0$ and $\sum_{|\alpha|\leq m}D^\alpha f_{\alpha, l}\rightarrow 0$ strongly in $H^{-m}(B(0,1/2))$. It follows from Theorem \ref{hom_thm_compactness2} that $u$ is a weak solution to \eqref{holder_eq_L0u}.
By letting $l\rightarrow \infty$, we obtain from \eqref{holder_ineq_contradict1} and \eqref{holder_ineq_contradict2} that
\begin{gather}\label{holder_ineq_contradict3}\Big(\fint_{B(0, 1)}|u|^2\Big)^{1/2}\leq 1,\\\label{holder_ineq_contradict4}\Big(\fint_{B(0, \theta)}|u-P^\theta(u)|^2\Big)^{1/2}+\theta^m\Big(\fint_{B(0, \theta)}|u|^2\Big)^{1/2}\geq\frac{\theta^{m-\sigma}}{2}, \end{gather}
where \eqref{holder_ineq_contradict3} is deduced from the weak convergence of $\{u_l\}$ in $L^2(B(0, 1))$, and property (iii) in Lemma \ref{pre_lem_poincare} is used to get \eqref{holder_ineq_contradict4}. These two inequalities, together with \eqref{holder_es_Lip}, yield that  $\frac{\theta^{m-\sigma}}{2}\leq2^{d/2}C_0\theta^m$, which contradicts the choice of $\theta$. This completes the proof.
\end{proof}

\begin{lemma}\label{holder_lem_iteration}
Let $\sigma\in(0, 1)$, $\varepsilon_0$ and $\theta$ be the constants given by Lemma \ref{holder_lem_onestep}. Let $u_\varepsilon\in H^m(B(0, 1); \mathbb{R}^n)$ be a weak solution to \eqref{holder_eq_L2} in $B(0, 1)$ and $\lambda\in [0, \varepsilon_0^2]$. If $0<\varepsilon<\varepsilon_0\theta^{k-1}$ for some $k\geq 1$, then
\begin{align}\label{holder_es_iteration}
&\quad\Big(\fint_{B(0, \theta^k)}|u_\varepsilon-P^{\theta^k}(u_\varepsilon)|^2\Big)^{1/2}+\theta^{km}\sqrt{\lambda}\Big(\fint_{B(0, \theta^k)}|u_\varepsilon|^2\Big)^{1/2}\nonumber\\
&\leq \theta^{k(m-\sigma)}\bigg\{\Big(\fint_{B(0, 1)}|u_\varepsilon-P^1(u_\varepsilon)|^2\Big)^{1/2}+\sqrt{\lambda}\Big(\fint_{B(0, 1)}|u_\varepsilon|^2\Big)^{1/2}+I_k\bigg\},
\end{align}
where $$I_k:=\varepsilon_0^{-1}\sum_{|\alpha|\leq m}\sum_{l=0}^{k-1}\theta^{l(m-|\alpha|+\sigma)}\Big(\fint_{B(0, \theta^l)}|f_\alpha|^2\Big)^{1/2}. $$
\end{lemma}
\begin{proof}
With Lemma \ref{holder_lem_onestep} at our disposal, \eqref{holder_es_iteration} follows from a standard induction argument on $k$ as \cite[Lemma 6.6]{Shen2018_Approximate}. We omit the details and just mention that rescaling and property (ii) in Lemma \ref{pre_lem_poincare} are used in the process.
\end{proof}

\begin{proof}[Proof of Theorem \ref{holder_thm_holder}]
Let $\varepsilon_0$ and $\theta$ be given by Lemma \ref{holder_lem_onestep}. For the case $\lambda\in[\varepsilon_0^2R^{-2m}, R^{-2m}]$, $u_\varepsilon$  satisfies $\mathcal{L}_\varepsilon^{\widetilde{A}}u_\varepsilon+\widetilde{\lambda}u_\varepsilon=\sum_{|\alpha|\leq m}D^\alpha \widetilde{f}_\alpha$ in $B$, where $\widetilde{A}=\varepsilon_0^2A$, $\widetilde{\lambda}=\varepsilon_0^2\lambda$, $\widetilde{f}_\alpha=\varepsilon_0^2f_\alpha$. If Theorem \ref{holder_thm_holder} holds for $\lambda\in[0, \varepsilon_0^2R^{-2m}]$, we will obtain that for $\varepsilon\leq r\leq R/2$,
\begin{align*}
&\quad\Big(\fint_{B(x_0, r)}|\nabla^{m} u_\varepsilon|^2\Big)^{1/2}+\sqrt{\widetilde{\lambda}}\Big(\fint_{B(x_0, r)}|u_\varepsilon|^2\Big)^{1/2}\nonumber\\
&\leq C_\sigma\Big(\frac{R}{r}\Big)^\sigma\bigg\{R^{-m}\Big(\fint_{B(x_0, R)}|u_\varepsilon|^2\Big)^{1/2}+\sum_{|\alpha|\leq m}\sup_{r\leq t\leq R}t^{m-|\alpha|}\Big(\fint_{B(x_0, t)}|\widetilde{f}_\alpha|^2\Big)^{1/2}\bigg\},
\end{align*}
where $C_\sigma$ depends only on $d, m ,n, \sigma$ and $\widetilde{A}$. This implies \eqref{holder_es_large} directly. Thus, in the following we suppose that $\lambda\in[0, \varepsilon_0^2R^{-2m}]$.

By translation and dilation, we may assume that $x_0=0$ and $R=1$. We claim that
\begin{align}
&\Big(\fint_{B(0, r)}|\nabla^{m} u_\varepsilon|^2\Big)^{1/2}+\sqrt{\lambda}\Big(\fint_{B(0, r)}|u_\varepsilon|^2\Big)^{1/2}\leq C_\sigma\Big(\frac{1}{r}\Big)^\sigma\bigg\{\Big(\fint_{B(0, 1)}|\nabla^{m} u_\varepsilon|^2\Big)^{1/2} \nonumber\\&\qquad \quad\quad\quad\quad\quad\quad\quad+\Big(\fint_{B(0, 1)}|u_\varepsilon|^2\Big)^{1/2}+\sum_{|\alpha|\leq m}\sup_{r\leq t\leq 1}t^{m-|\alpha|}\Big(\fint_{B(0, t)}|f_\alpha|^2\Big)^{1/2}\bigg\}, \label{holder_es_large_claim}
\end{align}if $\lambda\in[0, \varepsilon_0^2]$, $\varepsilon\leq r\leq 1$.
Obviously, \eqref{holder_es_large} follows from \eqref{holder_es_large_claim} and Cacciopolli's inequality \eqref{pre_ineq_cacci1}.
It remains to prove \eqref{holder_es_large_claim}. Since the case $r\geq\varepsilon_0\theta$ is trivial, we may assume that $r<\varepsilon_0\theta$. If $\varepsilon_0\theta^{k+1}\leq r<\varepsilon_0\theta^k$ for some $k\geq1$, we have
\begin{align*}
&\quad\Big(\fint_{B(0, r)}|\nabla^{m} u_\varepsilon|^2\Big)^{1/2}+\sqrt{\lambda}\Big(\fint_{B(0, r)}|u_\varepsilon|^2\Big)^{1/2}\\&\leq C\bigg\{\Big(\fint_{B(0, \theta^k/2)}|\nabla^{m} u_\varepsilon|^2\Big)^{1/2}+\sqrt{\lambda}\Big(\fint_{B(0, \theta^k/2)}|u_\varepsilon|^2\Big)^{1/2}\bigg\}\\
&\leq C\theta^{-k\sigma}\bigg\{\Big(\fint_{B(0, 1)}|u_\varepsilon-P^1(u_\varepsilon)|^2\Big)^{1/2}+\sqrt{\lambda}\Big(\fint_{B(0, 1)}|u_\varepsilon|^2\Big)^{1/2}+I_k \\&\qquad +\sum_{|\alpha|\leq m} \theta^{k(m-|\alpha|+\sigma)} \Big(\fint_{B(0, \theta^k)}|f_\alpha|^2\Big)^{1/2}\bigg\}\\
&\leq Cr^{-\sigma}\bigg\{\Big(\fint_{B(0, 1)}|\nabla^m u_\varepsilon|^2\Big)^{1/2}+\sqrt{\lambda}\Big(\fint_{B(0, 1)}|u_\varepsilon|^2\Big)^{1/2}+I_{k+1}\bigg\},
\end{align*}
where we have used Cacciopolli's inequality \eqref{pre_ineq_cacci_remark} and Lemma \ref{holder_lem_iteration} in the second inequality as well as property (iv) in Lemma \ref{pre_lem_poincare} in the last step.
Observing that $$I_k\leq C\sum_{|\alpha|\leq m}\sup_{r\leq t\leq 1}t^{m-|\alpha|}\Big(\fint_{B(0, t)}|f_\alpha|^2\Big)^{1/2},$$
we obtain \eqref{holder_es_large_claim} and complete the proof.
\end{proof}

\begin{remark}\label{holder_remark_B2}
Fixing $x_0=0$ in Theorem \ref{holder_thm_holder}, we could avoid translations of $A$ and employ Theorem \ref{hom_thm_compactness} instead of Theorem \ref{hom_thm_compactness2} in the compactness argument of Lemma \ref{holder_lem_onestep} without assuming that $A$ satisfies \eqref{hom_condition_Deltayz}. This results in \eqref{holder_es_large} with $x_0=0$ under the assumption that $A\in B^2$ and satisfies \eqref{intro_cond_bdd}--\eqref{intro_cond_el}.
\end{remark}

Now we establish a Liouville theorem for the higher-order elliptic systems with $B^2$-coefficients.

\begin{coro}
Assume that $A\in B^2$ satisfies \eqref{intro_cond_bdd}--\eqref{intro_cond_el}. Let $u\in H^m_{loc}(\mathbb{R}^d; \mathbb{R}^n)$ be a weak solution to $\mathcal{L}_1 u=0 \text{ in } \mathbb{R}^d$. Suppose that there exist a constant $C_u>0$ and some $\delta>0$ such that
\begin{align*}
\Big(\fint_{B(0, R)}|u|^2\Big)^{1/2}\leq C_uR^{m-\delta} ~\quad\textrm{ for any}~ R>1.
\end{align*}
Then $u\in \mathcal{P}_{m-1}$.
\end{coro}
\begin{proof}
Thanks to Remark \ref{holder_remark_B2}, for $1<r<R/2$ and any $\sigma\in(0, 1)$  we have
\begin{align*}\Big(\fint_{B(0, r)}|\nabla^m u|^2\Big)^{1/2}\leq C_\sigma\Big(\frac{R}{r}\Big)^\sigma R^{-m}\Big(\fint_{B(0, R)}|u|^2\Big)^{1/2}
\leq C_{\sigma, u}r^{-\sigma}R^{\sigma-\delta}.
\end{align*}
By choosing $\sigma<\delta$ and letting $R\rightarrow\infty$, we see that $\nabla^mu=0$ in $B(0, r)$. Since $r>1$ is arbitrary, it follows that $\nabla^mu=0$ in $\mathbb{R}^d$, which implies that  $u\in \mathcal{P}_{m-1}$.
\end{proof}

As an application of Theorem \ref{holder_thm_holder}, we obtain the following result.

\begin{theorem}\label{holder_thm_holder2}
Let $u$ be the solution to
\begin{equation*}(-1)^{m}\sum_{|\alpha|=|\beta|=m} D^\alpha(A^{\alpha\beta}D^\beta u)+T^{-2m}u=\sum_{|\alpha|\leq m} D^\alpha f_\alpha~~~\text{in}~\mathbb{R}^d, \end{equation*}
given by Proposition \ref{appcor_prop_1}, where $f_\alpha\in L^2_{\rm{loc, unif}}(\mathbb{R}^d; \mathbb{R}^n)$ for $|\alpha|\leq m$. Then for $2\leq q\leq q^+$, $\sigma\in(0, 1)$ and $1\leq r\leq T$, we have
\begin{align*}
\|\nabla^m u\|_{S^q_r}+T^{-m}\|u\|_{S^2_r}\leq C_\sigma\Big(\frac{T}{r}\Big)^\sigma\bigg\{\sum_{|\alpha|<m}\sup_{r\leq t\leq T}t^{m-|\alpha|}\|f_\alpha\|_{S^2_t}+\sum_{|\alpha|=m}\|f_\alpha\|_{S^q_r}\bigg\},
\end{align*}
where $q^+>2$ is given in Lemma \ref{pre_lem_reverse} and $C_\sigma$ depends only on $d, m, n, \sigma$ and $A$.
\end{theorem}
\begin{proof}
The case $T/2\leq r\leq T$ follows directly from \eqref{appcor_es_u}. For the case $1\leq r\leq T/2$, by Theorem \ref{holder_thm_holder} with $\varepsilon=1, \lambda=T^{-2m}, R=T$, and inequality \eqref{pre_ineq_reverse1}, we obtain for $x_0\in R^d$,
\begin{align*}
\Big(\fint_{B(x_0, r)}|\nabla^m u|^q\Big)^{1/q}+T^{-m}\Big(\fint_{B(x_0, r)}|u|^2\Big)^{1/2}\leq C_\sigma\Big(\frac{T}{r}\Big)^\sigma\bigg\{T^{-m}\Big(\fint_{B(x_0, T)}|u|^2\Big)^{1/2}\\ +\sum_{|\alpha|<m}\sup_{r\leq t\leq T}t^{m-|\alpha|}\Big(\fint_{B(x_0, t)}|f_\alpha|^2\Big)^{1/2}+\sum_{|\alpha|=m}\sup_{r\leq t\leq T}\Big(\fint_{B(x_0, t)}|f_\alpha|^q\Big)^{1/q}\bigg\},
\end{align*}
where $2\leq q\leq q^+$. Taking the supremum over $x_0\in \mathbb{R}^d$, it yields
\begin{align*}
\|\nabla^m u\|_{S^q_r}+T^{-m}\|u\|_{S^2_r}\leq& C_\sigma\Big(\frac{T}{r}\Big)^\sigma\bigg\{\sum_{|\alpha|<m}\sup_{r\leq t\leq T}t^{m-|\alpha|}\|f_\alpha\|_{S^2_t}+\sum_{|\alpha|=m}\sup_{r\leq t\leq T}\|f_\alpha\|_{S^q_t}\bigg\}\\
\leq& C_\sigma\Big(\frac{T}{r}\Big)^\sigma\bigg\{\sum_{|\alpha|<m}\sup_{r\leq t\leq T}t^{m-|\alpha|}\|f_\alpha\|_{S^2_t}+\sum_{|\alpha|=m}\|f_\alpha\|_{S^q_r}\bigg\},
\end{align*}
where Proposition \ref{appcor_prop_1} is used for the first inequality, and \eqref{hom_inequ_1} is used for the second.
\end{proof}

\begin{coro}\label{holder_coro_chi}
Let $T>1$ and $\sigma\in (0, 1)$. Then for any $1\leq r\leq T$,
\begin{gather}
\|\nabla^m \chi_T\|_{S^{q}_r}+T^{-m}\|\chi_T\|_{S^2_r}\leq C_\sigma\left(\frac{T}{r}\right)^\sigma,\label{holder_es_chi}\\
\|\nabla^m(\chi_T-\chi_{\widetilde{T}})\|_{S^{q}_r}+T^{-m}\|\chi_T-\chi_{\widetilde{T}}\|_{S^2_r}\leq C_\sigma\left(\frac{T}{r}\right)^\sigma\sup_{r\leq t\leq T}t^m\|T^{-2m}\chi_{\widetilde{T}}\|_{S^2_t},\label{holder_es_delta_chi}
\end{gather}
for any $\widetilde{T}\geq T$, where $2\leq q\leq q^+$ and $C_\sigma$ depends only on $d, m, n, \sigma$ and $A$.
\end{coro}
\begin{proof}
The results are obtained by applying Theorem \ref{holder_thm_holder2} to the equations of $\chi_T$ and $\chi_T-\chi_{\widetilde{T}}$ respectively, where $u=\chi_T-\chi_{\widetilde{T}}$ satisfies
\begin{equation}\label{holder_eq_delta_chi}(-1)^{m}\sum_{|\alpha|=|\beta|=m} D^\alpha(A^{\alpha\beta}D^\beta u)+T^{-2m}u=-(T^{-2m}-\widetilde{T}^{-2m})\chi_{\widetilde{T}}~~~\text{ in }~\mathbb{R}^d, \end{equation}
and $0\leq T^{-2m}-\widetilde{T}^{-2m}\leq T^{-2m}$.
\end{proof}


\section{A quantitative ergodic theorem in the higher-order case}\label{sec_ergodic}

In this part, we present a generalized quantitative ergodic theorem to bound $\|f\|_{S_1^2}$ with the higher-order differences and the $S_R^2$-norm of the derivatives of $f$ up to certain order.

Let $f\in L^1_{\rm{loc, unif}}(\mathbb{R}^d)$ and define
\begin{equation*}\omega_k(f; L, R):=\sup_{y_1\in \mathbb{R}^d}\inf_{|z_1|\leq L}\cdots\sup_{y_k\in \mathbb{R}^d}\inf_{|z_k|\leq L}\|\Delta_{y_1z_1}\cdots\Delta_{y_kz_k}(f)\|_{S^2_R},
\end{equation*}
where $0<L, R<\infty$ and $k\geq 1$. Throughout this section, we also define
\begin{equation}\label{ergodic_def_u}
u(x, t):=f*\Phi_t(x),
\end{equation}
where
$$\Phi_t(y)=t^{-d/2}\Phi(y/\sqrt{t})=c_d t^{-d/2}\exp(-|y|^2/(4t))$$
is the standard heat kernel. Therefore, $u$ satisfies the heat equation $\partial_t u=\Delta_x u$.

The following two lemmas were obtained in \cite{Armstrong2016_Bounded, Shen2018_Approximate}.

\begin{lemma}\label{ergodic_lem_1}
Let $u$ be defined as \eqref{ergodic_def_u}. Then for $0<R<\infty$,
\begin{equation*}
\|u\|_{S^2_R}\leq C\{\|u(\cdot, R^2)\|_\infty+R\|\nabla f\|_{S^2_R}\},
\end{equation*}
where $C$ depends only on $d$.
\end{lemma}

\begin{lemma}
Assume $f\in L^1_{\rm{loc, unif}}(\mathbb{R}^d)$ and $\langle f\rangle=0$. Let $u$ be defined as \eqref{ergodic_def_u}.
Then for every $k\in\mathbb{N}^+$, $0<R, L<\infty$ and $t\geq kR^2$,
\begin{gather}
\|u(\cdot, t)\|_{L^\infty}\leq C^k\Big\{\omega_k(f; L, R)+\exp\left(-\frac{ct}{kL^2}\right)\|f\|_{S^2_R}\Big\},\label{ergodic_es_infty1}\\
\|\nabla_x u(\cdot, t)\|_{L^\infty}\leq \frac{C^k}{\sqrt{t}}\Big\{\omega_k(f; L, R)+\exp\left(-\frac{ct}{kL^2}\right)\|f\|_{S^2_R}\Big\},\label{ergodic_es_infty2}
\end{gather}
where $C$ and $c$ depend only on $d$.
\end{lemma}

\begin{lemma}\label{ergodic_lem_u_infty}
Assume $f\in H^l_{loc}(\mathbb{R}^d)$ for some $l\in \mathbb{N}^+$ such that $\nabla^i f\in L^2_{\rm{loc, unif}}(\mathbb{R}^d)$, $i=0, 1, \cdots, l$, and $\langle f\rangle=0$. Let $n=\lceil\frac{l}{2}\rceil$, i.e. the smallest integer larger than $\frac{l}{2}$, and let $u$ be defined as \eqref{ergodic_def_u}. Then for every $k\in\mathbb{N}^+$ and $T>1$,
\begin{align}
\|u(\cdot, 1)\|_{L^\infty}\leq &C\sum_{i=0}^{n-1}T^{2i}\inf_{1\leq L\leq T}\Big\{\omega_k(\nabla^{2i}f; L, T)+\exp\left(-\frac{cT^2}{L^2}\right)\|\nabla^{2i}f\|_{S^2_T}\Big\}\nonumber\\
&\quad+C\int_1^T t^{l-1}\inf_{1\leq L\leq t}\Big\{\omega_k(\nabla^l f; L, t)+\exp\left(-\frac{ct^2}{L^2}\right)\|\nabla^l f\|_{S^2_t}\Big\},\label{ergodic_es_u_infty}
\end{align}
where $C$ and $c$ depend only on $d$ and $k$.
\end{lemma}
\begin{proof} Since $\partial_t u=\Delta_x u$,
by iteration we may deduce that
\begin{align}
\|u(\cdot, 1)\|_{L^\infty}&\leq \sum_{i=0}^{n-1}T^{2i}\|\partial_s^iu(\cdot, T^2)\|_{L^\infty}+\int_1^{T^2}s^{n-1}\|\partial_s^nu(\cdot, s)\|_{L^\infty} ds\nonumber\\
&\leq \sum_{i=0}^{n-1}T^{2i}\|\nabla_x^{2i}u(\cdot, T^2)\|_{L^\infty}+\int_1^{T^2}s^{n-1}\|\nabla_x^{2n}u(\cdot, s)\|_{L^\infty} ds.\label{ergodic_es_iteration}
\end{align}
  Note that $\langle\nabla^i f\rangle=0$ for $i=0, 1, \cdots, l$. For $i\leq n-1$, by \eqref{ergodic_es_infty1} with $T=cR$,
\begin{align}
\|\nabla_x^{2i}u(\cdot, T^2)\|_{L^\infty} \leq C\inf_{1\leq L\leq T}\Big\{\omega_k(\nabla^{2i}f; L, T)+\exp\left(-\frac{cT^2}{L^2}\right)\|\nabla^{2i}f\|_{S^2_T}\Big\}. \label{ergodic_es_2i}
\end{align}
To estimate $\|\nabla_x^{2n}u(\cdot, s)\|_{L^\infty}$, we divide the analysis into two cases. If $n=\frac{l}{2}$, we apply \eqref{ergodic_es_infty1} with $R=c\sqrt{s}$ to obtain directly
\begin{align*}
\|\nabla_x^{2n}u(\cdot, s)\|_{L^\infty}\leq C\inf_{1\leq L\leq \sqrt{s}}\Big\{\omega_k(\nabla^lf; L, \sqrt{s})+\exp\left(-\frac{cs}{L^2}\right)\|\nabla^lf\|_{S^2_{\sqrt{s}}}\Big\}.
\end{align*}
If $n=\frac{l+1}{2}$, we use the equality $$\nabla_x^{2n}u(\cdot, s)=\nabla_x(\nabla_x^lf*\Phi_t)$$ and \eqref{ergodic_es_infty2} with $R=c\sqrt{s}$ to obtain
\begin{align*}
\|\nabla_x^{2n}u(\cdot, s)\|_{L^\infty}\leq \frac{C}{\sqrt{s}}\inf_{1\leq L\leq \sqrt{s}}\Big\{\omega_k(\nabla^lf; L, \sqrt{s})+\exp\left(-\frac{cs}{L^2}\right)\|\nabla^lf\|_{S^2_{\sqrt{s}}}\Big\}.
\end{align*}
As a result, by setting $t=\sqrt{s}$,
\begin{align*}
\int_1^{T^2}s^{n-1}\|\nabla_x^{2n}u(\cdot, s)\|_{L^\infty} ds\leq C\int_1^T t^{l-1}\inf_{1\leq L\leq t}\Big\{\omega_k(\nabla^l f; L, t)+\exp\left(-\frac{ct^2}{L^2}\right)\|\nabla^l f\|_{S^2_t}\Big\},
\end{align*}
which, together with \eqref{ergodic_es_iteration} and \eqref{ergodic_es_2i}, gives \eqref{ergodic_es_u_infty}.
\end{proof}

\begin{theorem}\label{ergodic_thm_ergodic}
Let $f\in H^l_{loc}(\mathbb{R}^d)$ such that $\nabla^i f\in L^2_{\rm{loc, unif}}(\mathbb{R}^d)$, $i=0, 1, \cdots, l$, and $\langle f\rangle=0$. Then for any $k\geq1$ and $T\geq 2$,
\begin{align}
\|f\|_{S^2_1}&\leq C\sum_{i=0}^{l-1}T^i\inf_{1\leq L\leq T}\Big\{\omega_k(\nabla^i f; L, T)+\exp\left(-\frac{cT^2}{L^2}\right)\|\nabla^i f\|_{S^2_T}\Big\}\nonumber\\
&\quad+C\int_1^T t^{l-1}\inf_{1\leq L\leq t}\Big\{\omega_k(\nabla^l f; L, t)+\exp\left(-\frac{ct^2}{L^2}\right)\|\nabla^l f\|_{S^2_t}\Big\}dt,\label{ergodic_ineq_ergodic}
\end{align}
where $C$ depends only on $d, k$ and $l$, and $c$ depends only on $d$ and $k$.
\end{theorem}
\begin{proof}
Applying Lemma \ref{ergodic_lem_1} repeatedly, we have
\begin{align*}
\|f\|_{S^2_1}\leq C\sum_{i=0}^{l-1}\|\nabla^iu(\cdot, 1)\|_{L^\infty}+C\|\nabla^l f\|_{S_1^2},
\end{align*}
where $C$ depends only on $d$ and $l$. In view of the fact $$\nabla^i u(\cdot, 1)=\nabla^if*\Phi_1$$ with $\nabla^if\in H^{l-i}_{loc}(\mathbb{R}^d)$, it follows from Lemma \ref{ergodic_lem_u_infty} that
\begin{align}
\sum_{i=0}^{l-1}\|\nabla^iu(\cdot, 1)\|_{L^\infty}&\leq C\sum_{i=0}^{l-1}\sum_{j=0}^{\lceil\frac{l-i}{2}\rceil-1}T^{2j}\inf_{1\leq L\leq T}\Big\{\omega_k(\nabla^{2j}\nabla^if; L, T)+\exp\left(-\frac{cT^2}{L^2}\right)\|\nabla^{2j}\nabla^if\|_{S^2_T}\Big\}\nonumber\\
&\quad+Cl\int_1^T t^{l-1}\inf_{1\leq L\leq t}\Big\{\omega_k(\nabla^l f; L, t)+\exp\left(-\frac{ct^2}{L^2}\right)\|\nabla^l f\|_{S^2_t}\Big\}dt,\nonumber
\end{align}
where $C$ and $c$ depend only on $d$ and $k$. For $j\leq \lceil\frac{l-i}{2}\rceil-1$,
\begin{align*}
2j+i\leq l-i+1-2+i=l-1.
\end{align*}
Thus,
\begin{align}
\sum_{i=0}^{l-1}\|\nabla^iu(\cdot, 1)\|_{L^\infty}&\leq C\sum_{j=0}^{l-1}T^j\inf_{1\leq L\leq T}\Big\{\omega_k(\nabla^j f; L, T)+\exp\left(-\frac{cT^2}{L^2}\right)\|\nabla^j f\|_{S^2_T}\Big\}\nonumber\\
&\quad+C\int_1^T t^{l-1}\inf_{1\leq L\leq t}\Big\{\omega_k(\nabla^l f; L, t)+\exp\left(-\frac{ct^2}{L^2}\right)\|\nabla^l f\|_{S^2_t}\Big\}dt,\label{ergodic_es_f}
\end{align}
where $C$ depends only on $d, k$ and $l$. Finally, noticing that $\|\nabla^l f\|_{S^2_1}$ is bounded by the second integral in the r.h.s. of \eqref{ergodic_es_f} over the interval $[1, 2]$, we get \eqref{ergodic_ineq_ergodic}.
\end{proof}


\section{Estimates of approximate correctors}\label{sec_estimate}
In this section we establish some principal estimates for approximate correctors by using the large-scale H\"{o}lder estimates in Section \ref{sec_holder} and the quantitative ergodic theorem in Section \ref{sec_ergodic}. With these estimates in hand, we provide the proofs of Theorems \ref{estimate_thm_1} and \ref{estimate_thm_2}.

Recall that the difference operator for $y, z\in \mathbb{R}^d$ is defined by\begin{equation*}\Delta_{yz}f(x):=f(x+y)-f(x+z).
\end{equation*}
Let $$P=P_k=\{(y_1, z_1), (y_2, z_2), \dots, (y_k, z_k)\},$$
where $(y_i, z_i)\in\mathbb{R}^d\times\mathbb{R}^d$. We define the higher-order difference for $P$ by
\begin{equation*}
  \Delta_P(f):=\Delta_{y_1z_1}\cdots\Delta_{y_kz_k}(f)
\end{equation*}
(if $k=0$, then $P=\emptyset$ and $\Delta_P(f)=f$). Observe that
\begin{align*}
\Delta_P(fg)(x)=\sum_{Q\subset P}\Delta_Q(f)(x+z_{j_1}+\cdots+z_{j_t})\cdot\Delta_{P\setminus Q}(g)(x+y_{i_1}+\cdots+y_{i_l}),
\end{align*}
where the sum is taken over all $2^k$ subsets $Q=\{(y_{i_1}, z_{i_1}), \dots, (y_{i_l}, z_{i_l})\}$ of $P$, with $P\setminus Q=\{(y_{j_1}, z_{j_1}), \dots, (y_{j_t}, z_{j_t})\}$. Here, $i_1<\cdots<i_l$, $j_1<\cdots<j_t$, and $l+t=k$. By H\"{o}lder's inequality, this implies for $\frac{1}{r}\geq\frac{1}{p}+\frac{1}{q}$,
\begin{align*}
\|\Delta_P(fg)\|_{S^r_R}\leq \sum_{Q\subset P}\|\Delta_Q(f)\|_{S^p_R}\|\Delta_{P\setminus Q}(g)\|_{S^q_R}.
\end{align*}

To estimate $\|\nabla^l\chi_T\|_{S^2_1}$,  we follow  the idea of Theorem \ref{ergodic_thm_ergodic} to figure out $\omega_k(\nabla^l\chi_T; L, R)$. To this aim, we need the following lemma, which generalizes Theorem \ref{holder_thm_holder2} and Lemma \ref{appcor_lem_SR} in terms of higher-order differences.

\begin{lemma}\label{estimate_lem_difference}
Suppose that $A\in APW^2$ satisfies \eqref{intro_cond_bdd}--\eqref{intro_cond_el} and the assumptions of Theorem \ref{holder_thm_holder2} hold. Let $k\geq 0$, $P=P_k$, and let $q^+$ be given in Lemma \ref{pre_lem_reverse}. Then for any $1\leq r\leq T$ and $\sigma\in(0, 1)$,
\begin{align}
&\quad\|\Delta_P(\nabla^m u)\|_{S^q_r}+T^{-m}\|\Delta_P(u)\|_{S^2_r}\nonumber\\
&\leq C_\sigma\left(\frac{T}{r}\right)^{\sigma}\bigg\{\sum_{|\alpha|<m}\sup_{r\leq t\leq T}t^{m-|\alpha|}\|\Delta_P(f_\alpha)\|_{S^2_t}+\sum_{|\alpha|=m}\|\Delta_P(f_\alpha)\|_{S^{q_0}_r}\bigg\}\nonumber\\
&\quad+C_\sigma\left(\frac{T}{r}\right)^{\sigma}\sum_{P=Q_0\cup Q_1\cup\cdots\cup Q_l}\|\Delta_{Q_1}A\|_{S^p_r}\cdots\|\Delta_{Q_l}A\|_{S^p_r}\label{estimate_ineq_difference1}\\
&\qquad\cdot\bigg\{\sum_{|\alpha|<m}\sup_{r\leq t\leq T}t^{m-|\alpha|}\|\Delta_{Q_0}(f_\alpha)\|_{S^2_t}+\sum_{|\alpha|=m}\|\Delta_{Q_0}(f_\alpha)\|_{S^{q_0}_r}\bigg\},\nonumber
\end{align}
and for any $r\geq T$,
\begin{align}
&\|\Delta_P(\nabla^m u)\|_{S^q_r}+T^{-m}\|\Delta_P(u)\|_{S^2_r}\leq C\sum_{|\alpha|\leq m}T^{m-|\alpha|}\|\Delta_P(f_\alpha)\|_{S^{q_0}_r}\nonumber\\
&\qquad\quad+C\!\sum_{P=Q_0\cup Q_1\cup\cdots\cup Q_l}\|\Delta_{Q_1}A\|_{S^p_r}\cdots\|\Delta_{Q_l}A\|_{S^p_r}\Big(\sum_{|\alpha|\leq m}T^{m-|\alpha|}\|\Delta_{Q_0}(f_\alpha)\|_{S^{q_0}_r}\Big),\label{estimate_ineq_difference2}
\end{align}
where $2\leq q\leq q_0\leq q^+$, $\frac{1}{q}-\frac{1}{q_0}\geq \frac{k}{p}$, $C_\sigma$ depends only on $d, m, n, k, \sigma$ and $A$, and $C$ depends on $d, m, n, k$ and $\mu$. The sums in \eqref{estimate_ineq_difference1} and \eqref{estimate_ineq_difference2} are taken over all partitions of $P=Q_0\cup Q_1\cup\cdots\cup Q_l$ with $1\leq l\leq k-1$ and $Q_j\neq\emptyset$ for $j=1, \dots, l$.
\end{lemma}

\begin{proof}
The lemma is proved by an induction argument on $k$ based on Theorem \ref{holder_thm_holder2} and Lemma \ref{appcor_lem_SR}. Since the process is rather similar to the one of \cite[Lemma 8.1]{Shen2018_Approximate} for second-order elliptic systems, we omit the details for concision.
\end{proof}

Let $\rho_k(L, R)$ be defined as
\begin{align}\label{def_rho}
\rho_k(L, R)=\sup_{y_1\in \mathbb{R}^d}\inf_{|z_1|\leq L}\cdots\sup_{y_k\in \mathbb{R}^d}\inf_{|z_k|\leq L}\sum\|\Delta_{Q_1}(A)\|_{S^p_R}\cdots\|\Delta_{Q_l}(A)\|_{S^p_R},
\end{align}
where the sum is taken over all partitions of $P=Q_1\cup\cdots\cup Q_l$ with $1\leq l\leq k$, and $p$ is given by $\frac{k}{p}=\frac{1}{2}-\frac{1}{q^+}$, $q^+$ is the exponent in Lemma \ref{pre_lem_reverse}. We may assume $q^+\leq 2(k+1)$, thereby $q^+\leq p$.

\begin{coro}\label{estimate_es_om_rho}
Suppose that $A\in APW^2$ satisfies \eqref{intro_cond_bdd}--\eqref{intro_cond_el} and $T\geq 1$. Then for any $\sigma\in(0, 1)$, $k\geq 1$ and $0<L<\infty$, if $1\leq R\leq T$, we have
\begin{align*}
\sum_{l\leq m}T^{l-m}\omega_k(\nabla^l \chi_T; L, R)\leq C_\sigma\left(\frac{T}{R}\right)^\sigma\rho_k(L, R),
\end{align*}
where $C_\sigma$ depends only on $d, m, n, k, \sigma$ and $A$; if $R\geq T$, we have
\begin{align*}
\sum_{l\leq m}T^{l-m}\omega_k(\nabla^l \chi_T; L, R)\leq C\rho_k(L, R),
\end{align*}
where $C$ depend on $d, m, n, k$ and $\mu$.
\end{coro}

Now we are in a position to establish some further estimates on $\|\nabla^l\chi_T\|_{S^2_1}$, $0\leq l\leq m$.
\begin{proof}[Proof of Theorem \ref{estimate_thm_1}]
 Since \eqref{estimate_es_m_chi} follows from \eqref{holder_es_chi}, it is sufficient to prove \eqref{estimate_es_chi_S1}. By applying Theorem \ref{ergodic_thm_ergodic} to $\nabla^l \chi_T$, $0\leq l\leq m-1$, we get
\begin{align}
\|\nabla^l\chi_T\|_{S^2_1}&\leq C\sum_{i=0}^{m-l-1}T^i\inf_{1\leq L\leq T}\Big\{\omega_k(\nabla^i\nabla^l\chi_T; L, T)+\exp\left(-\frac{cT^2}{L^2}\right)\|\nabla^i \nabla^l\chi_T\|_{S^2_T}\Big\}\nonumber\\&\quad+C\int_1^T t^{m-l-1}\inf_{1\leq L\leq t}\Big\{\omega_k(\nabla^{m-l}\nabla^l\chi_T; L, t)+\exp\left(-\frac{ct^2}{L^2}\right)\|\nabla^{m-l}\nabla^l\chi_T\|_{S^2_t}\Big\}dt\nonumber\\
&\leq CT^{m-l}\inf_{1\leq L\leq T}\Big\{\rho_k(L, T)+\exp\left(-\frac{cT^2}{L^2}\right)\Big\}\nonumber\\&\quad+C\int_1^T t^{m-l-1}\inf_{1\leq L\leq t}\Big\{\rho_k(L, t)+\exp\left(-\frac{ct^2}{L^2}\right)\Big\}\left(\frac{T}{t}\right)^\sigma dt,\label{estimate_es_chi_S1_proof}
\end{align}
where we have used Corollary \ref{estimate_es_om_rho}, \eqref{appcor_es_chi_SR} and \eqref{holder_es_chi}. Since the first term in the r.h.s. of \eqref{estimate_es_chi_S1_proof} is bounded by the last integral in \eqref{estimate_es_chi_S1_proof} from $T/2$ to $T$, we obtain \eqref{estimate_es_chi_S1} immediately.
\end{proof}

Under  additional conditions on $\rho_k,$ it is possible to establish estimates similar to \eqref{estimate_es_m_chi} on $\nabla^l\chi_T$ for $0\leq l\leq m-1$.
\begin{coro}\label{estimate_coro_l}
Let $0\leq l\leq m-1$. Suppose there exist some $k\geq 1$ and $\theta\geq0$ such that \eqref{estimate_condition_rho} holds. Then for any $\vartheta>\max(0, m-l-\theta)$, $T\geq 1$,
\begin{align}\label{niuchi}
\|\nabla^l\chi_T\|_{S^2_1}\leq C_\vartheta T^\vartheta,
\end{align}
where $C_\vartheta$ depends only on $d, m, n, k, l, \theta, \vartheta$ and $A$.
\end{coro}
\begin{proof}
Let $T\geq 2$. By choosing $L=t^\delta$ in \eqref{estimate_es_chi_S1} with $\delta\in(0, 1)$,
\begin{align*}
\|\nabla^l\chi_T\|_{S^2_1}&\leq C_\sigma T^\sigma\int_1^T t^{m-l-1-\sigma}\left\{t^{-\theta\delta}+\exp\left(-ct^{2(1-\delta)}\right)\right\} dt\\
&\leq C_\sigma T^\sigma\int_1^T t^{m-l-1-\sigma-\theta\delta}dt+C_{\sigma, \delta}T^\sigma,
\end{align*}
where $C_{\sigma, \delta}$ depends only on $d, m, n, k, \sigma, \delta$ and $A$.
If $\theta>m-l-1$, we can choose $\sigma>m-l-\theta$ arbitrarily, and $\delta$ close enough to $1$ depending on $\theta, \sigma, m, l$,  such that  $\theta\delta+\sigma>m-l$. We then obtain for any $\sigma>\max(0, m-l-\theta)$, $$\|\nabla^l\chi_T\|_{S^2_1}\leq C_\sigma T^\sigma.$$
If $\theta\leq m-l-1$, then $\theta\delta+\sigma<m-l$ for any $\sigma, \delta\in (0, 1)$. Direct calculations imply that,
$$\|\nabla^l\chi_T\|_{S^2_1}\leq \frac{C_\sigma}{m-l-\sigma-\theta\delta}T^{m-l-\theta\delta}+C_{\sigma, \delta}T^\sigma. $$
Since $\delta$ is arbitrary and $\sigma<m-l-\theta\delta$, we obtain for any $\vartheta>m-l-\theta\geq 1$,
$$\|\nabla^l\chi_T\|_{S^2_1}\leq C_\vartheta T^\vartheta.$$
Therefore the proof of \eqref{niuchi} is done for $T\geq2$. The estimate \eqref{niuchi} for $1\leq T\leq 2$ follows easily from the one for $T=2$, and the proof is thus completed.
\end{proof}

\begin{proof}[Proof of Theorem \ref{estimate_thm_2}]
Thanks to \eqref{estimate_es_m_chi} and Corollary \ref{estimate_coro_l}, under the assumption \eqref{estimate_condition_rho}  with $\theta>m$, we have
\begin{align}
\|\nabla^l \chi_T\|_{S^2_1}\leq C_\sigma T^\sigma,\label{estimate_es_Tsigma}
\end{align}
for any $\sigma\in(0, 1)$, $T\geq 1$ and $0\leq l\leq m$.
Set $g=\chi_T-\chi_{\widetilde{T}}$, where $T\leq \widetilde{T}\leq 2T$. By Corollary \ref{holder_coro_chi} and \eqref{estimate_es_Tsigma}, we obtain  for any $T\geq 1$ and $\sigma\in (0, 1)$,
\begin{align}
\sum_{l\leq m}T^{l-m}\|\nabla^lg\|_{S^2_1}\leq C_\sigma T^{\sigma-m}.\label{estimate_es_lg}
\end{align}
This means that, for $1\leq l\leq m$, $\{\nabla^l \chi_T\}$ is convergent with respect to the norm $\|\cdot\|_{S^2_1}$ as $T\rightarrow \infty$, and thereby $\|\nabla^l\chi_T\|_{S^2_1}\leq C$ for $T\geq 1$. Likewise, if there exists some $\vartheta>0$  such that  for any $T\geq 1$,
\begin{equation}\label{estimate_claim_g}
\|g\|_{S^2_1}\leq C_\vartheta T^{-\vartheta},
\end{equation}
then we can conclude that $\{\chi_T\}$ is also convergent with respect to the norm $\|\cdot\|_{S^2_1}$ as $T\rightarrow \infty$, and $\|\chi_T\|_{S^2_1}\leq C$. Denote the limit of $\chi_T$ as $\chi$. Since $\nabla^l \chi_T\in APW^2$ and $\|g\|_{W^2}\leq \|g\|_{S^2_1}$, we obtain that $\nabla^l\chi\in APW^2$ for each $0\leq l\leq m$, which obviously satisfies
\begin{equation*}
 \sum_{|\alpha|=|\beta|=m} D^\alpha(A^{\alpha\beta}D^\beta u)=-\sum_{|\alpha|=|\beta|=m} D^\alpha(A^{\alpha\beta}D^\beta P)~\quad\text{in}~\mathbb{R}^d.
\end{equation*}

Therefore, it remains to show \eqref{estimate_claim_g}. Let $u(x, t)=g*\Phi_t(x)$. In view of Lemma \ref{ergodic_lem_1} and \eqref{estimate_es_lg}, it is sufficient to prove \begin{align*}
\|u(\cdot, 1)\|_{L^\infty}\leq CT^{-\vartheta}.
\end{align*}
By \eqref{ergodic_es_infty1} and the fact that $g\in APW^2$, we know that $\|u(\cdot, t)\|_{L^\infty}\rightarrow 0$ as $t\rightarrow \infty$. Therefore,
\begin{align}
\|u(\cdot, 1)\|_{L^\infty}&\leq \int_1^\infty\|\partial_t u(\cdot, t)\|_{L^\infty} dt
 \leq \int_1^\infty\|\nabla_x (\nabla g*\Phi_t)\|_{L^\infty} dt\nonumber\\
&\leq Ct_0\|\nabla g\|_{S^2_1}+\int_{t_0^2}^\infty\|\nabla_x (\nabla g*\Phi_t)\|_{L^\infty} dt\nonumber\\
&\leq Ct_0T^{\sigma-1}+\int_{t_0^2}^\infty\Big\{\|\nabla_x (\nabla \chi_T*\Phi_t)\|_{L^\infty}+\|\nabla_x (\nabla \chi_{\widetilde{T}}*\Phi_t)\|_{L^\infty} \Big\} dt,\label{estimate_es_u_infty}
\end{align}
where $t_0>1$ is to be determined later and we have used estimate (\cite{Armstrong2016_Bounded})
\begin{equation*}
\|\nabla_x (\nabla g*\Phi_t)\|_{L^\infty}\leq Ct^{-1/2}\|\nabla g\|_{S^2_1} \quad\textrm{ for any } t\geq 1
\end{equation*}
and \eqref{estimate_es_lg} for the last two steps, respectively. Since $T\leq \widetilde{T}\leq 2T$ and $\|\nabla \chi_T\|_{S^2_1}\leq C$, using    \eqref{ergodic_es_infty2} and the change of variables, the second term in the r.h.s. \eqref{estimate_es_u_infty} can be bounded by
\begin{align}
C\int_{t_0}^\infty \inf_{1\leq L\leq t}\Big\{T^{m-1+\sigma}\rho_k(L, t)+\exp(-\frac{ct^2}{L^2})\Big\} dt.\label{estimate_es_boundness}
\end{align}
Now choose $L=t^\delta$ with $\delta\in(0, 1)$. If $\delta$ is large enough such that $\theta\delta>1$, then \eqref{estimate_es_boundness} may be bounded by
\begin{align*}
C\int_{t_0}^\infty \left\{T^{m-1+\sigma}t^{-\theta\delta}+\exp(-ct^{2(1-\delta)})\right\} dt\leq C_{\delta, \sigma}T^{m-1+\sigma}t_0^{1-\theta\delta},
\end{align*}
which, together with \eqref{estimate_es_u_infty},   gives
\begin{align*}
\|u(\cdot, 1)\|_{L^\infty}\leq Ct_0T^{\sigma-1}+C_{\delta, \sigma}T^{m-1+\sigma}t_0^{1-\theta\delta}.
\end{align*}
To determine $t_0$, let $t_0T^{\sigma-1}=T^{m-1+\sigma}t_0^{1-\theta\delta}$, i.e. $t_0=T^{\frac{m}{\theta\delta}}$. Then $\|u(\cdot, 1)\|_{L^\infty}\leq CT^{\frac{m}{\theta\delta}+\sigma-1}$,
where $\frac{m}{\theta\delta}+\sigma-1<0$ if $\delta$ is close to $1$ and $\sigma$ is small enough (since $\theta>m$). This completes the proof.
\end{proof}

\begin{coro}\label{estimate_coro_psi}
Suppose the assumptions of Theorem \ref{estimate_thm_2} hold. Then
\begin{align*}
\|\nabla^m \chi_T-\psi\|_{S^2_T}\leq CT^{-m},
\end{align*}
where $\psi$ is the solution of equation \eqref{hom_eq_corrector} in $\mathcal{V}^n$.
\end{coro}
\begin{proof}
Applying Lemma \ref{appcor_lem_SR} to \eqref{holder_eq_delta_chi}, we obtain
\begin{align}\label{estimate_difference}
\|\nabla^m(\chi_T-\chi_{\widetilde{T}})\|_{S^2_R}\leq CT^{-m}\|\chi_{\widetilde{T}}\|_{S^2_R}\leq CT^{-m}
\end{align}
for any $R\geq T$ and $\widetilde{T}\geq T$, where Theorem \ref{estimate_thm_2} was used. As shown in Theorem \ref{estimate_thm_2}, $\nabla^m \chi_{\widetilde{T}}\rightarrow \nabla^m \chi$ with respect to the norm $\|\cdot\|_{S^2_1}$ as $\widetilde{T}\rightarrow \infty$. By \eqref{appcor_conver} and the uniqueness of $\psi$ in $\mathcal{V}^n$, $\psi=\nabla^m \chi \in APW^2(\mathbb{R}^d)$. Thus  by letting $\widetilde{T}\rightarrow\infty$ in \eqref{estimate_difference},
\begin{align*}
\|\nabla^m\chi_T-\psi\|_{S^2_R}\leq CT^{-m},
\end{align*}
which gives the desired result.
\end{proof}


\section{Estimates on the dual approximate correctors}\label{sec_dual}

For $1\leq i, j\leq n$, $|\alpha|=|\beta|=m$, let $B^{\alpha\beta}_{T}=(B^{\alpha\beta}_{T, ij})$ be defined as
\begin{align*}
B^{\alpha\beta}_{T, ij}(y):=A^{\alpha\beta}_{ij}(y)+\sum_{|\gamma|=m}A^{\alpha\gamma}_{ik}(y)D^\gamma\chi_{T, kj}^\beta(y)-\widehat{A}^{\alpha\beta}_{ij}.
\end{align*}
We introduce the dual approximate correctors $\phi_T=(\phi_{T, ij}^{\alpha\beta})$ as the solution to
\begin{align}(-\Delta)^m \phi_{T, ij}^{\alpha\beta}+T^{-2m}\phi^{\alpha\beta}_{T, ij}=B^{\alpha\beta}_{T, ij}-\langle B^{\alpha\beta}_{T, ij}\rangle \quad\textrm{in } \mathbb{R}^d.\label{dual_eq_phi}
\end{align}

\begin{lemma}\label{dual_lem_SR}
Let $u$ be the weak solution to
$$(-\Delta)^{m}u+T^{-2m}u=f\quad\textrm{in}~\mathbb{R}^d,$$
given by Proposition \ref{appcor_prop_1}, where $f\in L^2_{\rm{loc, unif}}(\mathbb{R}^d)$. Then for any $0<R<\infty$,
\begin{gather*}
T^{-m}\|\nabla^m u\|_{S^2_R}+T^{-2m}\|u\|_{S^2_R}\leq C\|f\|_{S^2_R},\quad
\|\nabla^{2m}u\|_{S^2_R}\leq C\log\Big(2+\frac{T}{R}\Big)\|f\|_{S^2_R},
\end{gather*}
where $C$ depends only on $m$ and $d$.
\end{lemma}
\begin{proof}
The proof for the case $d\geq2m$ and $d$ is odd follows from the estimate on the fundamental solution $\Gamma(x)$ of the operator $(-\Delta)^m+1$ in $\mathbb{R}^d$ with pole at the origin (see e.g. \cite{Erdelyi1953, Boyling1996}),   $$|\nabla^l\Gamma(x)|\leq C|x|^{2m-d-l}e^{-c|x|}, \quad l\geq0,$$ and some singular integral estimates. As the analysis is almost the same as \cite[Lemma 9.2]{Shen2018_Approximate}, let us omit the details.  The proof for the other cases follows from the method of descending, i.e., introducing dummy variables and considering the equations in $\mathbb{R}^d$ with $d\geq2m$ and $d$ odd.
\end{proof}

\begin{lemma}\label{dual_lem_phi_SR}
Assume that $A\in APW^2$ and  satisfies \eqref{intro_cond_bdd}--\eqref{intro_cond_el}. Let $T>1$, $\sigma\in(0, 1)$. Then for any $0\leq l\leq 2m$, if $1\leq R\leq T$,
\begin{align}\|\nabla^l\phi_T\|_{S^2_R}\leq C_\sigma T^{2m-l}\left(\frac{T}{R}\right)^\sigma,\label{dual_es_phi_SR1}
\end{align}
where $C_\sigma$ depends only on $d, m, n, \sigma$ and $A$; if $R\geq T$,
\begin{align}\|\nabla^l\phi_T\|_{S^2_R}\leq CT^{2m-l},\label{dual_es_phi_SR2}
\end{align}
where $C$ depends only on $d, m, n, \mu$.
\end{lemma}
\begin{proof}
Note that  by \eqref{holder_es_chi} and \eqref{appcor_es_chi_SR},
\begin{equation}\label{niu4}
\begin{cases}
\|A^{\alpha\beta}+\sum_{|\gamma|=m}A^{\alpha\gamma}D^\gamma\chi_{T}^\beta\|_{S^2_R}\leq C_\sigma\left(\frac{T}{R}\right)^\sigma  &\text{ for any }1\leq R\leq T, \sigma\in(0, 1),\\
\|A^{\alpha\beta}+\sum_{|\gamma|=m}A^{\alpha\gamma}D^\gamma\chi_{T}^\beta\|_{S^2_R}\leq C  &\text{ for any } R\geq T,
\end{cases}
\end{equation}
 where $C_\sigma$ depends only on $d, m, n, \sigma, A$, and $C$ depends only on $d, m, n, \mu$.
Applying Lemma \ref{dual_lem_SR} to equation \eqref{dual_eq_phi}, we obtain that
 \begin{equation}\label{niu5}
\begin{cases} \sum_{l\leq 2m}T^{-2m+l}\|\nabla^l\phi_T\|_{S^2_R}\leq C_\sigma\left(\frac{T}{R}\right)^\sigma\|B_T-\langle B_T\rangle\|_{S^2_R} &\text{ for } 1\leq R\leq T,\\
\sum_{l\leq 2m}T^{-2m+l}\|\nabla^l\phi_T\|_{S^2_R}\leq C\|B_T-\langle B_T\rangle\|_{S^2_R} &\text{ for } R\geq T.
 \end{cases}
\end{equation}
Combining \eqref{niu4}, \eqref{niu5} and the fact \begin{align*}\|B_T-\langle B_T\rangle\|_{S^2_R}\leq 2\|A^{\alpha\beta}+\sum_{|\gamma|=m}A^{\alpha\gamma}D^\gamma\chi_{T}^\beta\|_{S^2_R},\end{align*}we obtain \eqref{dual_es_phi_SR1} and \eqref{dual_es_phi_SR2} immediately.
\end{proof}

\begin{lemma}\label{dual_lem_phi_om}
Assume that $A\in APW^2$ and satisfies \eqref{intro_cond_bdd}--\eqref{intro_cond_el}. Let $T>1$, $k\geq 1$ and $\sigma\in(0, 1)$. Then for any $0\leq l\leq 2m$,
\begin{align}\omega_k(\nabla^l\phi_T; L, R)\leq C_\sigma T^{2m-l}\left(\frac{T}{R}\right)^\sigma\rho_k(L, R) \quad \text{ for } 1\leq R\leq T, \label{dual_es_phi_om1}
\end{align}
where $C_\sigma$ depends only on $d, m, n, \sigma$ and $A$, and
\begin{align}\omega_k(\nabla^l\phi_T; L, R)\leq CT^{2m-l}\rho_k(L, R) \quad \text{ for } R\geq T,\label{dual_es_phi_om2}
\end{align}
 where $C$ depends only on $d, m, n, \mu$.
\end{lemma}
\begin{proof}
Since the proofs of \eqref{dual_es_phi_om1} and \eqref{dual_es_phi_om2} are rather similar, we only provide the details for the one of \eqref{dual_es_phi_om1}.
Applying the operator $\Delta_P$ to equation \eqref{dual_eq_phi}, we have
\begin{align*}(-\Delta)^m \Delta_P\phi_{T}^{\alpha\beta}+T^{-2m}\Delta_P\phi^{\alpha\beta}_{T}=\Delta_P B^{\alpha\beta}_{T} \quad\textrm{ in } \mathbb{R}^d,
\end{align*}
which, in view of Lemma \ref{dual_lem_SR}, implies that, for $1\leq R\leq T$ and $\sigma\in(0, 1)$,
\begin{align*}\sum_{l\leq 2m}T^{-2m+l}\|\Delta_P\nabla^l\phi_T\|_{S^2_R}\leq C_\sigma\left(\frac{T}{R}\right)^\sigma\|\Delta_P B_T\|_{S^2_R},
\end{align*}
 where $C_\sigma$ depends only on $d, m, n, \sigma$ and $A$. As a result, for any $1\leq R\leq T$,
\begin{align}\sum_{l\leq 2m}T^{-2m+l}\omega_k(\nabla^l\phi_T; L, R)\leq C_\sigma\left(\frac{T}{R}\right)^\sigma\omega_k(B_T; L, R).\label{dual_es_om_phi}
\end{align}
It remains to calculate $\omega_k(B_T; L, R)$. Let $p, q$ satisfy $\frac{k}{p}=\frac{1}{2}-\frac{1}{q^+}$ and $\frac{k-1}{p}+\frac{1}{q}=\frac{1}{2}$. Noticing that
$$\Delta_P B_T=\Delta_P A+\sum_{Q\subset P}\Delta_P A\cdot\Delta_P(\nabla^m\chi_T),$$
it follows from H\"{o}lder's inequality that
\begin{align*}
\|\Delta_P B_T\|_{S^2_R}\leq \|\Delta_P A\|_{S^2_R}+\|A\|_{L^\infty}\|\Delta_P(\nabla^m \chi_T)\|_{S^2_R}+\sum_{Q\subset P, Q\neq \emptyset}\|\Delta_Q A\|_{S^p_R}\|\Delta_{P\setminus Q}(\nabla^m\chi_T)\|_{S^q_R},
\end{align*}
from which and  Lemma \ref{estimate_lem_difference}, we obtain
\begin{align*}
\|\Delta_P B_T\|_{S^2_R}\leq C_\sigma\left(\frac{T}{R}\right)^\sigma\sum_{Q_1\cup\cdots\cup Q_l=P}\|\Delta_{Q_1} A\|_{S^p_R}\cdots\|\Delta_{Q_l} A\|_{S^p_R}.
\end{align*}
 This yields, for any $1\leq R\leq T$ and $\sigma\in(0, 1)$,
\begin{align}
\omega_k(B_T; L, R)\leq C_\sigma\left(\frac{T}{R}\right)^\sigma\rho_k(L, R). \label{dual_es_om_BT}
\end{align}
Combining \eqref{dual_es_om_phi} and \eqref{dual_es_om_BT}, we get \eqref{dual_es_phi_om1} and complete the proof.
\end{proof}

Following the arguments in Section \ref{sec_estimate}, we may obtain some further estimates on $\phi_T$.
\begin{theorem}\label{dual_thm_1}
Assume that $A\in APW^2$ and satisfies \eqref{intro_cond_bdd}--\eqref{intro_cond_el}. Let $k\geq 1$, $\sigma\in(0, 1)$. Then for any $T\geq 2$ and $0\leq l<2m$,
\begin{align}
\|\nabla^l\phi_T\|_{S^2_1}\leq C_\sigma\int_1^T t^{2m-l-1}\inf_{1\leq L\leq t}\Big\{\rho_k(L, t)+\exp\left(-\frac{ct^2}{L^2}\right)\Big\}\left(\frac{T}{t}\right)^\sigma dt,\label{dual_es_phi_S11}
\end{align}
where $C_\sigma$ depends only on $d, m, n, k, \sigma$ and $A$, and $c$ depends only on $d$ and $k$. Furthermore, if condition \eqref{estimate_condition_rho} is satisfied for some $k\geq1$ and $\theta>m$, then for any $T\geq 1$,
\begin{gather}\label{dual_es_phi_S12}
\begin{split}
&\|\nabla^l\phi_T\|_{S^2_1}\leq C   \quad\quad \textrm{ if }  m\leq l\leq 2m,\\
&\|\nabla^l\phi_T\|_{S^2_1}\leq CT^{m-l}  \quad\quad\textrm{ if } l<m.
\end{split}
\end{gather}
\end{theorem}
\begin{proof}
We give a sketch of the proof here and refer readers to Section \ref{sec_estimate} for more details. Similar to Theorem \ref{estimate_thm_1}, \eqref{dual_es_phi_S11} follows from Theorem \ref{ergodic_thm_ergodic} and Lemmas \ref{dual_lem_phi_SR}, \ref{dual_lem_phi_om}.

To prove \eqref{dual_es_phi_S12}, we note that by \eqref{dual_es_phi_S11} and direct computations,
$$\|\nabla^l\phi_T\|_{S^2_1}\leq C_\vartheta T^\vartheta \quad\text{ for any } \vartheta>\max\{0, 2m-l-\theta\}.$$
Let $g:=\phi_T-\phi_{\widetilde{T}}$ with $T\leq \widetilde{T}\leq 2T$. Then applying Lemma \ref{dual_lem_SR} to the equation of $g$ and using the estimate \eqref{estimate_es_lg}, we obtain that
\begin{align}
\|\nabla^l g\|_{S^2_1}\leq C_\sigma T^{m-l+\sigma} \quad\text{for } l\leq 2m.\label{dual_es_g1}
\end{align}
  Furthermore, by the argument used in the proof of \eqref{estimate_claim_g}, \eqref{dual_es_g1} can be improved as
\begin{align}
\|\nabla^l g\|_{S^2_1}\leq C_\sigma T^{m-l-\sigma} \quad\text{ for } l<2m. \label{dual_es_g2}
\end{align}
 Now the desired results follow from \eqref{dual_es_g1} and \eqref{dual_es_g2}.
\end{proof}

By the definition of $\chi_T$, we know that $$\sum_{|\alpha|=m}D^\alpha B_{T, ij}^{\alpha\beta}=(-1)^{m+1}T^{-2m}\chi_{T, ij}^\beta.$$
Setting \begin{align}
h_{T, ij}^{\beta}=\sum_{|\alpha|=m}D^\alpha\phi_{T, ij}^{\alpha\beta},\label{dual_def_h}
\end{align}
\eqref{dual_eq_phi} implies that $h_T=(h_{T, ij}^\beta)$ satisfies the equation
$$(-\Delta)^m h_{T, ij}^\beta+T^{-2m}h^\beta_{T, ij}=(-1)^{m+1}T^{-2m}\chi_{T, ij}^\beta \quad \textrm{ in } \mathbb{R}^d.$$
Thanks to Lemma \ref{dual_lem_SR}, we have the following theorem.
\begin{theorem}\label{dual_thm_2} Let $h_T$ be defined as in \eqref{dual_def_h}. Then the following estimate holds with $C$ depending only on $m$ and $d,$
$$T^m\|\nabla^mh_T\|_{S^2_1}\leq C\|\chi_T\|_{S^2_1}.$$
\end{theorem}


\section{Convergence rates}\label{sec_conver}

Let $\zeta\in C_c^\infty(B(0, 1))$ be a nonnegative function with $\int_{\mathbb{R}^d}\zeta=1$, and $\zeta_\varepsilon(x)=\varepsilon^{-d}\zeta(x/\varepsilon)$. Define $$S_\varepsilon(f)(x)=\zeta_\varepsilon*f(x)=\int_{\mathbb{R}^d}\zeta_\varepsilon(y)f(x-y)dy. $$
Note that, for $1\leq p\leq \infty$,
\begin{equation}\label{conver_ineq_S1}
\|S_\varepsilon (f)\|_{L^p(\mathbb{R}^d)}\leq \|f\|_{L^p(\mathbb{R}^d)}.
\end{equation}
It is known that (see e.g., \cite{Shen2017_Boundary}) if $f\in L^p(\mathbb{R}^d)$ and $g\in L^p_{\rm{loc, unif}}(\mathbb{R}^d)$, then
\begin{equation}\label{conver_ineq_S2}
\|g(x/\varepsilon)S_\varepsilon(f)\|_{L^p(\mathbb{R}^d)}\leq \sup_{x\in\mathbb{R}^d}\bigg(\fint_{B(x, 1)}|g|^p\bigg)^{1/p}\|f\|_{L^p(\mathbb{R}^d)},
\end{equation}
and for $f\in W^{1, p}(\mathbb{R}^d)$,
\begin{equation}\label{conver_ineq_S3}
\|S_\varepsilon(f)-f\|_{L^p(\mathbb{R}^d)}\leq C\varepsilon\|\nabla f\|_{L^p(\mathbb{R}^d)},
\end{equation}
where $C$ depends only on $d$.

Recall that $\Omega\subset\mathbb{R}^d$ is a bounded Lipschitz domain. Let $\delta\geq 2\varepsilon$ be a small parameter to be determined and $\eta_\delta\in C_c^\infty(\Omega)$ be a cut-off function satisfying that,
\begin{gather*}
  \eta_\delta=0 \text{ in } \Omega_\delta:=\{x\in \Omega: \textrm{dist}(x, \partial\Omega)<\delta\}, \eta_\delta=1 \text{ in } \Omega\setminus\Omega_{2\delta},\\0\leq \eta_\delta \leq 1, |\nabla^l\eta_\delta|\leq C\delta^{-l}, l\leq m. 
\end{gather*}
Define $K_{\varepsilon, \delta}f(x)=S_\varepsilon(\eta_\delta f)(x)$. Then $\mathrm{supp}(K_{\varepsilon, \delta}f)\subset\Omega\setminus\Omega_\varepsilon$, since $\delta\geq 2\varepsilon$.

\begin{lemma}[\cite{Shen2017_Convergence}]\label{conver_lem_Omve}
Let $\Omega$ be a bounded Lipschitz domain. Then for any $u\in H^1(\mathbb{R}^d)$,
$$\int_{\Omega_\varepsilon}|u|^2\leq C\varepsilon\|u\|_{H^1(\mathbb{R}^d)}\|u\|_{L^2(\mathbb{R}^d)},$$
where $C$ depends only on $\Omega$.
\end{lemma}

\begin{lemma}\label{conver_lem_om1}
Assume that $\Omega$ is a bounded Lipschitz domain and $A\in APW^2$ satisfies \eqref{intro_cond_bdd}--\eqref{intro_cond_el}. Let $u_\varepsilon$ be the weak solution to Dirichlet problem \eqref{intro_eq1} and $u_0$ be the solution to problem \eqref{intro_eq_hom}. Suppose further $u_0\in H^{m+1}(\Omega; \mathbb{R}^n)$. Set
$$\omega_\varepsilon:=u_\varepsilon-u_0-\varepsilon^m\sum_{|\gamma|=m}\chi_T^\gamma(x/\varepsilon)K_{\varepsilon, \delta}(D^\gamma u_0),$$
with $2\varepsilon\leq \delta\leq2$.
Then for any $\varphi\in H^m_0(\Omega; \mathbb{R}^n)$, we have
\begin{align}
&\quad\bigg|\sum_{|\alpha|=|\beta|=m}\int_\Omega A^{\alpha\beta}(x/\varepsilon)D^\beta \omega_\varepsilon D^\alpha\varphi\bigg|\nonumber\\
&\leq C\Big\{\delta+T^{-m}\|\chi_T\|_{S^2_1}+T^{-2m}\|\phi_T\|_{S^2_1}+\varepsilon\sum_{l<m} \big[\|\nabla^{m+l}\phi_T\|_{S^2_1}+\|\nabla^l\chi_T\|_{S^2_1}\big] \nonumber\\
&\quad+\|\nabla^m\chi_T-\psi\|_{B^2}\Big\}\cdot\big\{\|\nabla^m \varphi\|_{L^2(\Omega)}+\delta^{-1/2}\|\nabla^m \varphi\|_{L^2(\Omega_{3\delta})}\big\}\|\nabla^mu_0\|_{H^1(\Omega)},\label{conver_es_omvp1}
\end{align}
and
\begin{align}
&\quad\bigg|\sum_{|\alpha|=|\beta|=m}\int_\Omega A^{\alpha\beta}(x/\varepsilon)D^\beta \omega_\varepsilon D^\alpha\varphi\bigg|\nonumber\\
&\leq C\Big\{\delta+T^{-m}\|\chi_T\|_{S^2_1}+T^{-2m}\|\phi_T\|_{S^2_1}+\|\nabla^m\chi_T-\psi\|_{B^2}+\varepsilon\sum_{l<m} \big[\|\nabla^{m+l}\phi_T\|_{S^2_1}+\|\nabla^l\chi_T\|_{S^2_1}\big]\Big\}\nonumber\\
&\quad\cdot\{\|\nabla^{m+1} u_0\|_{L^2(\Omega\setminus\Omega_\delta)}+\delta^{-1}\|\nabla^m u_0\|_{L^2(\Omega_{2\delta})}+\|\nabla^m u_0\|_{L^2(\Omega)}\}\|\nabla^m \varphi\|_{L^2(\Omega)},\label{conver_es_omvp2}
\end{align}
where $C$ depends only on $m, n, A$ and $\Omega$.
\end{lemma}
\begin{proof}
Using the equations of $u_\varepsilon$ and $u_0$, a direct computation shows that for any $\varphi\in H^m_0(\Omega; \mathbb{R}^n)$,
\begin{align}
&\sum_{|\alpha|=|\beta|=m}\int_\Omega A^{\alpha\beta}(x/\varepsilon)D^\beta \omega_\varepsilon D^\alpha\varphi =\sum_{|\alpha|=|\beta|=m}\int_\Omega (\widehat{A}^{\alpha\beta}-A^{\alpha\beta}(x/\varepsilon))(D^\beta u_0-K_{\varepsilon, \delta}(D^\beta u_0))D^\alpha \varphi\nonumber\\
&\qquad\qquad-\sum_{|\alpha|=|\beta|=m}\int_\Omega B^{\alpha\beta}_T(x/\varepsilon)K_{\varepsilon, \delta}(D^\beta u_0)D^\alpha \varphi\nonumber\\
&\qquad\qquad-\sum_{|\alpha|=|\beta|=m}\int_\Omega A^{\alpha\beta}(x/\varepsilon)\sum_{|\gamma|=m}\sum_{\substack{\beta_1+\beta_2=\beta\\\beta_1<\beta}}\varepsilon^{m-|\beta_1|}D^{\beta_1}\chi_T^\gamma(x/\varepsilon)D^{\beta_2}K_{\varepsilon, \delta}(D^\gamma u_0)D^\alpha \varphi\nonumber\\
&\qquad\doteq I_1+I_2+I_3,\label{conver_iden_I}
\end{align}
where $$B^{\alpha\beta}_T(y):=A^{\alpha\beta}(y)+\sum_{|\gamma|=m}A^{\alpha\gamma}(y)D^\gamma \chi_T^\beta(y)-\widehat{A}^{\alpha\beta},$$
and $\beta_1<\beta$ means $\beta_1$ is a subindex of $\beta$, i.e. there exists a multi-index $\beta'\neq0$ such that $\beta=\beta_1+\beta'$. We will bound these three terms one by one.

First, observe that $$K_{\varepsilon, \delta}(D^\beta u_0)=S_\varepsilon(D^\beta u_0) \text{ in } \Omega\backslash\Omega_{2\delta+\varepsilon}.$$
 Thus,
\begin{align*}
D^\beta u_0-K_{\varepsilon, \delta}(D^\beta u_0)=D^\beta u_0 1_{\Omega_{2\delta+\varepsilon}}+(D^\beta u_0-S_\varepsilon(D^\beta u_0))1_{\Omega\backslash\Omega_{2\delta+\varepsilon}}-S_\varepsilon(\eta_\delta D^\beta u_0)1_{\Omega_{2\delta+\varepsilon}} ~\textrm{ in } \Omega.
\end{align*}
This implies that
\begin{align*}
|I_1|&\leq C\|\nabla^m u_0-S_\varepsilon(\nabla^m u_0)\|_{L^2(\Omega\backslash\Omega_{2\delta+\varepsilon})}\|\nabla^m \varphi\|_{L^2(\Omega)}+C\|\nabla^m u_0\|_{L^2(\Omega_{3\delta})}\|\nabla^m\varphi\|_{L^2(\Omega_{2\delta+\varepsilon})}\nonumber\\
&\leq C\varepsilon\|\nabla^{m+1} u_0\|_{L^2(\Omega)}\|\nabla^m \varphi\|_{L^2(\Omega)}+C\|\nabla^m u_0\|_{L^2(\Omega_{3\delta})}\|\nabla^m\varphi\|_{L^2(\Omega_{2\delta+\varepsilon})},
\end{align*}
where we have used the fact that $$S_\varepsilon(\eta_\delta D^\beta u_0)=S_\varepsilon(\eta_\delta D^\beta u_01_{\Omega_{3\delta}}) \quad\text{ in } \Omega_{2\delta+\varepsilon},$$
as well as inequality \eqref{conver_ineq_S1}  for the first step, and \eqref{conver_ineq_S3} for the last. 

To deal with $I_2$, we deduce by the equation of $\phi_T$,
\begin{align*}
I_2&=-\sum_{|\alpha|=|\beta|=m}\int_\Omega \big[((-\Delta)^m\phi_T^{\alpha\beta})(x/\varepsilon)+T^{-2m}\phi_T^{\alpha\beta}(x/\varepsilon)+\langle B^{\alpha\beta}_T\rangle\big]K_{\varepsilon, \delta}(D^\beta u_0)D^\alpha \varphi\nonumber\\
&=-\sum_{|\alpha|=|\beta|=m}\int_\Omega [(-1)^m\sum_{|\gamma|=m}\varepsilon^mD^\gamma(D^\gamma\phi_T^{\alpha\beta}(x/\varepsilon)-D^\alpha\phi_T^{\gamma\beta}(x/\varepsilon))\nonumber\\
&\quad+(-1)^m \sum_{|\gamma|=m}D^\gamma D^\alpha \phi_T^{\gamma\beta}(x/\varepsilon)+T^{-2m}\phi_T^{\alpha\beta}(x/\varepsilon)+\langle B^{\alpha\beta}_T\rangle]K_{\varepsilon, \delta}(D^\beta u_0)D^\alpha \varphi\nonumber\\
&=(-1)^{m}\varepsilon^m\sum_{|\alpha|=|\beta|=|\gamma|=m}\int_\Omega\sum_{\substack{\gamma_1+\gamma_2=\gamma\\\gamma_1<\gamma}}D^{\gamma_1} (D^\gamma\phi_T^{\alpha\beta}(x/\varepsilon)-D^\alpha\phi_T^{\gamma\beta}(x/\varepsilon))D^{\gamma_2}K_{\varepsilon, \delta}(D^\beta u_0)D^\alpha \varphi\nonumber\\
&\quad+(-1)^{m+1} \sum_{|\alpha|=|\beta|=m}\int_\Omega D^\alpha h_T^\beta(x/\varepsilon)K_{\varepsilon, \delta}(D^\beta u_0)D^\alpha \varphi\nonumber\\&\quad-T^{-2m}\sum_{|\alpha|=|\beta|=m}\int_\Omega\phi_T^{\alpha\beta}(x/\varepsilon)K_{\varepsilon, \delta}(D^\beta u_0)D^\alpha \varphi-\sum_{|\alpha|=|\beta|=m}\int_\Omega\langle B^{\alpha\beta}_T\rangle K_{\varepsilon, \delta}(D^\beta u_0)D^\alpha \varphi\nonumber\\
&\doteq I_{21}+I_{22}+I_{23}+I_{24},
\end{align*}
where $h_T$ is defined by \eqref{dual_def_h} and the fact that $D^\gamma\phi_T^{\alpha\beta}-D^\alpha\phi_T^{\gamma\beta}$ is skew-symmetric with respect to $(\alpha, \gamma)$ is used in the last step. Thanks to \eqref{conver_ineq_S2} and Theorem \ref{dual_thm_2}, $I_{22}$, $I_{23}$ can be bounded by,
\begin{gather*}
|I_{22}|\leq C\|\nabla^m h_T\|_{S^2_1}\|\nabla^m u_0\|_{L^2(\Omega)}\|\nabla^m \varphi\|_{L^2(\Omega)}\leq CT^{-m}\|\chi_T\|_{S^2_1}\|\nabla^m u_0\|_{L^2(\Omega)}\|\nabla^m \varphi\|_{L^2(\Omega)},\\
|I_{23}|\leq CT^{-2m}\|\phi_T\|_{S^2_1}\|\nabla^m u_0\|_{L^2(\Omega)}\|\nabla^m \varphi\|_{L^2(\Omega)}.
\end{gather*}
Since for any multi-index $\gamma$ with $|\gamma|\geq 1$, we have
\begin{align*}
D^{\gamma}K_{\varepsilon, \delta}(D^\beta u_0)=\varepsilon^{-|\gamma|+1}(D^{\gamma'}\zeta)_\varepsilon*D^{\gamma''}(D^\beta u_0\eta_\delta)=\varepsilon^{-|\gamma|+1}(D^{\gamma'}\zeta)_\varepsilon*[D^{\beta+\gamma''} u_0\eta_\delta+D^\beta u_0 D^{\gamma''}\eta_\delta]
\end{align*}
where $\gamma'+\gamma''=\gamma$ and $|\gamma''|=1$. In view of \eqref{conver_ineq_S2} and  $\textrm{supp}[(D^{\gamma'}\zeta)_\varepsilon*(D^\beta u_0 D^{\gamma''}\eta_\delta)]\subset\Omega_{2\delta+\varepsilon}$,  we get
\begin{align*}
|I_{21}|&\leq C\sum_{l<m}\varepsilon\|\nabla^{m+l}\phi_T\|_{S^2_1}\big[\|\nabla^{m+1} u_0\|_{L^2(\Omega\setminus\Omega_{\delta})}\|\nabla^m \varphi\|_{L^2(\Omega)}\\&\quad+\delta^{-1}\|\nabla^m u_0\|_{L^2(\Omega_{2\delta})}\|\nabla^m \varphi\|_{L^2(\Omega_{2\delta+\varepsilon})}\big].
\end{align*}
By the definition of $\widehat{A}$, we have \begin{align*}|\langle B_T^{\alpha\beta}\rangle|=|\langle \sum_{|\gamma|=m}A^{\alpha\gamma}(D^\gamma \chi_T^\beta-\psi^{\gamma\beta})\rangle|\leq C\|\nabla^m\chi_T-\psi\|_{B^2}.\end{align*}
Therefore, $$|I_{24}|\leq C\|\nabla^m\chi_T-\psi\|_{B^2} \|\nabla^m u_0\|_{L^2(\Omega)}\|\nabla^m \varphi\|_{L^2(\Omega)},$$  which, combined with the estimates on $I_{21}, I_{22}, I_{23}$, implies that
\begin{align}\label{conver_es_I2}
&\quad|I_2|\nonumber\\&\leq C\sum_{l<m}\varepsilon\|\nabla^{m+l}\phi_T\|_{S^2_1}\big[\|\nabla^{m+1} u_0\|_{L^2(\Omega\setminus\Omega_\delta)}\|\nabla^m \varphi\|_{L^2(\Omega)}+\delta^{-1}\|\nabla^m u_0\|_{L^2(\Omega_{2\delta})}\|\nabla^m \varphi\|_{L^2(\Omega_{2\delta+\varepsilon})}\big]\nonumber\\
&\quad+C\big[T^{-m}\|\chi_T\|_{S^2_1}+T^{-2m}\|\phi_T\|_{S^2_1}+\|\nabla^m\chi_T-\psi\|_{B^2}\big]\|\nabla^m u_0\|_{L^2(\Omega)}\|\nabla^m \varphi\|_{L^2(\Omega)}.
\end{align}

Finally, $I_3$ is essentially similar to $I_{21}$ and thus can be bounded as follows,
\begin{gather}
|I_3|\leq C\sum_{l<m}\varepsilon\|\nabla^l\chi_T\|_{S^2_1}[\|\nabla^{m+1} u_0\|_{L^2(\Omega\setminus\Omega_\delta)}\|\nabla^m \varphi\|_{L^2(\Omega)}\nonumber\\
+\delta^{-1}\|\nabla^m u_0\|_{L^2(\Omega_{2\delta})}\|\nabla^m \varphi\|_{L^2(\Omega_{2\delta+\varepsilon})}].\label{conver_es_I3}
\end{gather}
Taking the estimates on $I_1$--$I_3$ into \eqref{conver_iden_I} and using Lemma \ref{conver_lem_Omve}, we get \eqref{conver_es_omvp1} immediately.

To see \eqref{conver_es_omvp2}, we give a different estimate on $I_1$. Since  for any $\beta$,
\begin{align*}
D^\beta u_0-K_{\varepsilon, \delta}(D^\beta u_0)=\eta_\delta D^\beta u_0-S_\varepsilon(\eta_\delta D^\beta u_0)+(1-\eta_\delta)D^\beta u_0 \quad \textrm{ in } \Omega,
\end{align*}
it follows that
\begin{align}
|I_1|&\leq C\|\eta_\delta\nabla^m u_0-S_\varepsilon(\eta_\delta\nabla^m u_0)\|_{L^2(\Omega)}\|\nabla^m \varphi\|_{L^2(\Omega)}+C\|\nabla^m u_0\|_{L^2(\Omega_{2\delta})}\|\nabla^m\varphi\|_{L^2(\Omega_{2\delta})}.\label{conver_es_I11}
\end{align}
Thanks to \eqref{conver_ineq_S3}, we have
\begin{align*}
\|\eta_\delta\nabla^m u_0-S_\varepsilon(\eta_\delta\nabla^m u_0)\|_{L^2(\Omega)}&\leq C\varepsilon\|\nabla(\eta_\delta\nabla^m u_0)\|_{L^2(\mathbb{R}^d)}\nonumber\\
&\leq C\varepsilon\|\nabla^{m+1} u_0\|_{L^2(\Omega\setminus\Omega_\delta)}+C\|\nabla^m u_0\|_{L^2(\Omega_{2\delta})},
\end{align*}
which, together with \eqref{conver_es_I11}, gives
\begin{gather}
|I_1|\leq C[\varepsilon\|\nabla^{m+1} u_0\|_{L^2(\Omega\setminus\Omega_\delta)}+C\|\nabla^m u_0\|_{L^2(\Omega_{2\delta})}]\|\nabla^m \varphi\|_{L^2(\Omega)}\nonumber\\
+C\|\nabla^m u_0\|_{L^2(\Omega_{2\delta})}\|\nabla^m\varphi\|_{L^2(\Omega_{2\delta})}.\label{conver_es_I1_2}
\end{gather}
Now \eqref{conver_es_omvp2} follows directly from \eqref{conver_iden_I}, \eqref{conver_es_I2}, \eqref{conver_es_I3} and  \eqref{conver_es_I1_2}.
\end{proof}

\begin{lemma}
Suppose that the assumptions of Lemma \ref{conver_lem_om1} hold and $\varepsilon= T^{-m}$. Then
\begin{align}
\|\omega_\varepsilon\|_{H^m_0(\Omega)}\leq C\delta^{1/2}\|u_0\|_{H^{m+1}(\Omega)},\label{conver_es_Hm1}
\end{align}
with
\begin{align}
\delta=2T^{-m}+T^{-2m}\|\phi_T\|_{S^2_1}+\|\nabla^m\chi_T-\psi\|_{B^2}+T^{-m}\sum_{l<m} \big[\|\nabla^{m+l}\phi_T\|_{S^2_1}+\|\nabla^l\chi_T\|_{S^2_1}\big].\label{conver_def_de}
\end{align}
Suppose in addition $f\in H^{-m+1}(\Omega; \mathbb{R}^n)$, $\dot{g}\in W\!A^{m,2}(\partial\Omega; \mathbb{R}^n)$, and $A=A^*$ if $n\geq2$. Then
\begin{align}
\|\omega_\varepsilon\|_{H^m_0(\Omega)}\leq C\delta^{1/2}\{\|f\|_{H^{-m+1}(\Omega)}+\|\dot{g}\|_{W\!A^{m,2}(\partial\Omega)}\},\label{conver_es_Hm2}
\end{align}
where $\delta$ is given by \eqref{conver_def_de}.
\end{lemma}
\begin{proof}
Note that $\omega_\varepsilon\in H^m_0(\Omega; \mathbb{R}^n)$, and $\delta\leq C(d, m, n, A)$ by \eqref{appcor_conver}, \eqref{estimate_es_chi_S1} and \eqref{dual_es_phi_S11}. Obviously, \eqref{conver_es_Hm1} is a consequence of \eqref{conver_es_omvp1} by letting $\varphi=\omega_\varepsilon$ and $\delta$ be given as \eqref{conver_def_de}.

To prove \eqref{conver_es_Hm2}, we set $\varphi=\omega_\varepsilon$ in \eqref{conver_es_omvp2}, and it suffices to estimate $\|\nabla^{m+1} u_0\|_{L^2(\Omega\setminus\Omega_\delta)}$ and $\|\nabla^m u_0\|_{L^2(\Omega_{2\delta})}$. To this end, we decompose $u_0$ into $u_0=u_1+u_2$, such that $u_1$ solves
\begin{equation*}
\begin{cases}
\mathcal{L}_0 u_1 = f  &\text{ in } B,\\
Tr (D^\gamma u_1)=0  & \text{ on } \partial B, \text{ for  } 0\leq|\gamma|\leq m-1,
\end{cases}
\end{equation*}
where $B$ is a ball satisfying $\Omega\subset B$ and $f$ has been extended to $0$ outside $\Omega$. Following the arguments of Theorem 3.1 in \cite{Niu2018_Convergence}, we can obtain
\begin{gather}
\|\nabla^mu_0\|_{L^2(\Omega_{2\delta})} \leq C \delta^{1/2}\left\{\|\dot{g}\|_{W\!A^{m,2}(\partial\Omega)}+\|f\|_{H^{-m+1}(\Omega)}\right\},\label{conver_es_u0pa}\\
\|\nabla^{m+1} u_0\|_{L^2(\Omega\setminus\Omega_\delta)}\leq C\delta^{-1/2}\left\{\|\dot{g}\|_{W\!A^{m,2}(\partial\Omega)}+\|f\|_{H^{-m+1}(\Omega)}\right\},\nonumber
\end{gather}
which implies the desired result directly. 
\end{proof}

Now we are prepared to prove Theorem \ref{conver_thm_2} through the duality method inspired by  \cite{Suslina2013_Dirichlet,Suslina2013_Neumann}.
\begin{theorem}\label{conver_thm_main}
Suppose that the assumptions of Lemma \ref{conver_lem_om1} hold, $\varepsilon= T^{-m}$ and, in addition, $A=A^*$ if $n\geq2$. Then
\begin{align}
\|u_\varepsilon-u_0\|_{H^{m-1}_0(\Omega)}\leq C\delta\Big[\sum_{l\leq m}\|\nabla^{l}\chi_T\|_{S^2_1}+1\Big]\|u_0\|_{H^{m+1}(\Omega)},\label{conver_es_ratede}
\end{align}
where $\delta$ is given by \eqref{conver_def_de} and $C$ depends only on $m, n, A$ and $\Omega$.
\end{theorem}
\begin{proof}
Note that
\begin{align}
&\quad\varepsilon^m\Big\|\sum_{|\gamma|=m}\chi_T^\gamma(x/\varepsilon) K_{\varepsilon, \delta}(D^\gamma u_0)\Big\|_{H_0^{m-1}(\Omega)}\nonumber\\
&=\varepsilon^{m}\Big\|\sum_{l_1+l_2\leq m-1}\sum_{|\gamma|=m}\varepsilon^{-l_1-l_2}\nabla^{l_1}\chi_T^\gamma(x/\varepsilon) (\nabla^{l_2}\zeta)_\varepsilon*(D^\gamma u_0\eta_\delta)\Big\|_{L^2(\Omega)}\nonumber\\
&\leq C\varepsilon\sum_{l\leq m-1}\|\nabla^{l}\chi_T\|_{S^2_1}\|\nabla^m u_0\|_{L^2(\Omega)}\nonumber\\
&\leq C\delta\|\nabla^m u_0\|_{L^2(\Omega)},\label{conver_es_KuHm1}
\end{align}
where we have used \eqref{conver_ineq_S2} in the second step. In view of the definition of $\omega_\varepsilon$, it suffices to prove that
\begin{align*}
\|\omega_\varepsilon\|_{H^{m-1}_0(\Omega)}\leq C\delta\|u_0\|_{H^{m+1}(\Omega)}.
\end{align*}
To do this, consider the Dirichlet problem for $\varepsilon\geq 0$
\begin{equation*}
\begin{cases}
\mathcal{L}_\varepsilon v_\varepsilon =F  &\text{ in } \Omega,\\
Tr (D^\gamma v_\varepsilon)=0  & \text{ on } \partial\Omega, \text{ for  } 0\leq|\gamma|\leq m-1,
\end{cases}
\end{equation*}
with $F\in H^{-m+1}(\Omega)$. By setting $$\varpi:=v_\varepsilon-v_0-\varepsilon^m\sum_{|\gamma|=m}\chi_T^\gamma(x/\varepsilon)K_{\varepsilon, \delta}(D^\gamma v_0),$$
we have
\begin{align}
&\langle \omega_\varepsilon, F\rangle_{H_0^{m-1}(\Omega)\times H^{-m+1}(\Omega)}=\sum_{|\alpha|=|\beta|=m}\int_{\Omega} A^{\alpha\beta}(x/\varepsilon)D^\beta \omega_\varepsilon D^\alpha v_\varepsilon\nonumber\\
&=\sum_{|\alpha|=|\beta=m}\int_{\Omega}A^{\alpha\beta}(x/\varepsilon)D^\beta\omega_\varepsilon D^\alpha \varpi_\varepsilon+\sum_{|\alpha|=|\beta|=m}\int_{\Omega}A^{\alpha\beta}(x/\varepsilon)D^\beta\omega_\varepsilon D^\alpha v_0\nonumber\\
&\quad+\sum_{|\alpha|=|\beta|=m}\int_{\Omega}A^{\alpha\beta}(x/\varepsilon)D^\beta\omega_\varepsilon D^\alpha\Big[\varepsilon^m\sum_{|\gamma|=m}\chi_T^\gamma(x/\varepsilon)K_{\varepsilon, \delta}(D^\gamma v_0)\Big]\nonumber\\
&\doteq \mathcal{I}_1+\mathcal{I}_2+\mathcal{I}_3.\label{conver_es_omveF}
\end{align}
Thanks to \eqref{conver_es_Hm1} and \eqref{conver_es_Hm2},
\begin{align*}
|\mathcal{I}_1|\leq C\|\nabla^m \omega_\varepsilon\|_{L^2(\Omega)}\|\nabla^m \varpi_\varepsilon\|_{L^2(\Omega)}\leq C\delta\|u_0\|_{H^{m+1}(\Omega)}\|F\|_{H^{-m+1}(\Omega)}.
\end{align*}
By \eqref{conver_es_omvp1} and \eqref{conver_es_u0pa} for $v_0$, i.e.,
\begin{align}
\|\nabla^m v_0\|_{L^2(\Omega_{3\delta})}\leq C\delta^{1/2}\|F\|_{H^{-m+1}(\Omega)},\label{conver_es_v0pa}
\end{align}
we get
\begin{align*}
|\mathcal{I}_2|\leq C\delta\|u_0\|_{H^{m+1}(\Omega)}\|F\|_{H^{-m+1}(\Omega)}.
\end{align*}
To deal with $\mathcal{I}_3$, similar to \eqref{conver_es_KuHm1}, we deduce from \eqref{conver_es_omvp1} that
\begin{align*}
|\mathcal{I}_3|&\leq C\delta\|u_0\|_{H^{m+1}(\Omega)}\cdot\bigg\{\Big\|\nabla^m\Big[\varepsilon^m\sum_{|\gamma|=m}\chi_T^\gamma(x/\varepsilon)K_{\varepsilon, \delta}(D^\gamma v_0)\Big]\Big\|_{L^2(\Omega)}\\&\qquad+\delta^{-1/2}\Big\|\nabla^m\Big[\varepsilon^m\sum_{|\gamma|=m}\chi_T^\gamma(x/\varepsilon)K_{\varepsilon, \delta}(D^\gamma v_0)\Big]\Big\|_{L^2(\Omega_{3\delta})}\bigg\}\\
&\leq C\delta\|u_0\|_{H^{m+1}(\Omega)}\cdot\bigg\{\sum_{l\leq m}\|\nabla^{l}\chi_T\|_{S^2_1}\|\nabla^m v_0\|_{L^2(\Omega)}+\delta^{-1/2}\sum_{l\leq m}\|\nabla^{l}\chi_T\|_{S^2_1}\|\nabla^m v_0\|_{L^2(\Omega_{4\delta})}\bigg\}\\
&\leq C\delta\Big[\sum_{l\leq m}\|\nabla^{l}\chi_T\|_{S^2_1}\Big]\|u_0\|_{H^{m+1}(\Omega)}\|F\|_{H^{-m+1}(\Omega)},
\end{align*}
where we have used \eqref{conver_es_v0pa} in the last step.
In view of the estimates of $\mathcal{I}_1, \mathcal{I}_2, \mathcal{I}_3$ and \eqref{conver_es_omveF}, we have proved that
\begin{align*}
\Big|\langle \omega_\varepsilon, F\rangle_{H_0^{m-1}(\Omega)\times H^{-m+1}(\Omega)}\Big|\leq C\delta\Big[\sum_{l\leq m}\|\nabla^{l}\chi_T\|_{S^2_1}+1\Big]\|u_0\|_{H^{m+1}(\Omega)}\|F\|_{H^{-m+1}(\Omega)},
\end{align*}
which, by duality, yields the desired estimate \eqref{conver_es_ratede}. The proof is completed.
\end{proof}

\begin{remark}
Without the symmetry assumption on $A$, if $\partial\Omega\in C^{m,1}$, by the standard estimate for higher-order elliptic systems with constant coefficients, we  have
$$\|v_0\|_{H^{m+1}(\Omega)}\leq C\|F\|_{H^{-m+1}(\Omega)}$$
for $v_0$ in the proof of Theorem \ref{conver_thm_main}. Thus, by using a similar duality argument, we obtain
$$\|u_\varepsilon-u_0\|_{H^{m-1}_0(\Omega)}\leq C\{\delta+\delta^*\}\|u_0\|_{H^{m+1}(\Omega)},$$
where $\delta$ is given by \eqref{conver_def_de} and $\delta^*$ is given by \eqref{conver_def_de} with $A$ replaced by $A^*$, and $C$ depends only on $m, n, A$ and $\Omega$. See \cite{Shen2018_Approximate} for the second-order case.
\end{remark}

\begin{proof}[Proof of Theorem \ref{conver_thm_2}]
Note that by Theorem \ref{estimate_thm_1}, $$\sum_{l\leq m}\|\nabla^{l}\chi_T\|_{S^2_1}\leq C_\sigma \sum_{l\leq m-1}\Theta_{k, l, \sigma}(T),$$where $\Theta_{k, l, \sigma}(T)$ denotes the integral in the r.h.s. of \eqref{estimate_es_chi_S1}.
Similarly, thanks to Theorems \ref{estimate_thm_1}, \ref{dual_thm_1} and \ref{dual_thm_2}, we have $$\delta\leq C_\sigma\Big\{\|\nabla^m \chi_T-\psi\|_{B^2}+T^{-m}\sum_{l\leq m-1}\Theta_{k, l, \sigma}(T)\Big\},$$
for any $k\geq 1$ and $\sigma\in(0, 1)$, which, together with Theorem \ref{conver_thm_main}, gives \eqref{conver_es_rate1}.

If condition \eqref{estimate_condition_rho} holds for some $k\geq1$ and $\theta>m$,  by Theorems \ref{estimate_thm_2}, \ref{dual_thm_1} and \ref{dual_thm_2} we have
\begin{align*}
\|\nabla^m h_T\|_{S^2_1}+T^{-2m}\|\phi_T\|_{S^2_1}+\varepsilon\sum_{l<m} \big[\|\nabla^{m+l}\phi_T\|_{S^2_1}+\|\nabla^l\chi_T\|_{S^2_1}\big]\leq C\varepsilon,
\end{align*}
with $\varepsilon=T^{-m}$. Moreover, according to Corollary \ref{estimate_coro_psi},
\begin{align*}
\|\nabla^m\chi_T-\psi\|_{B^2}\leq CT^{-m}.
\end{align*}
As a result, it follows from Theorem \ref{conver_thm_main} that,
\begin{align*}
\|u_\varepsilon-u_0\|_{H^{m-1}_0(\Omega)}\leq C\varepsilon\|u_0\|_{H^{m+1}(\Omega)},
\end{align*}
which is exactly \eqref{intro_es_converrate}. The proof is completed.
\end{proof}

We end up with the proof of Theorem \ref{intro_thm_4}.

\begin{proof}[Proof of Theorem \ref{intro_thm_4}]
For $\dot{g}\in W\!A^{m, 2}(\partial\Omega; \mathbb{R}^n)$ and $f\in H^{-m}(\Omega;\mathbb{R}^n)$, let $u_\varepsilon$ satisfy $$\mathcal{L}^{A}_\varepsilon(u_\varepsilon)=f  \quad\text{ in }  \Omega,  \quad   u_\varepsilon=\dot{g}  \quad\text{ on } \partial\Omega$$ and let $u_0$ be the solution to the corresponding homogenized problem and $u_0\in H^{m+1}(\Omega)$.
Also let $\widetilde{u}_\varepsilon$ and $\widetilde{u}_0$ be the solutions to   $$\mathcal{L}^{\widetilde{A}}_\varepsilon(\widetilde{u}_\varepsilon)=f  \quad\text{ in }  \Omega,  \quad   \widetilde{u}_\varepsilon=\dot{g}  \quad\text{ on } \partial\Omega$$   and its homogenized problem, respectively.

Obviously, \eqref{pertur_cond_pertur} implies that $\|A-\widetilde{A}\|_{B^p}=0$, i.e., $A$ and $\widetilde{A}$ are in the same equivalence class, which gives $\widehat{A}=\widehat{\widetilde{A}}$ and $u_0=\widetilde{u}_0$.
Set $\omega_\varepsilon:=u_\varepsilon-\widetilde{u}_\varepsilon$. Then $\omega_\varepsilon\in H^{m}_0(\Omega)$ satisfies
\begin{align*}
\mathcal{L}^A_\varepsilon\omega_\varepsilon=(-1)^{m+1}\sum_{|\alpha|=|\beta|=m} D^\alpha[(A^{\alpha\beta}(\frac{x}{\varepsilon})-\widetilde{A}^{\alpha\beta}(\frac{x}{\varepsilon}))D^\beta\widetilde{u}_\varepsilon] \quad\textrm{ in }~\Omega.
\end{align*}
For $F\in H^{-m+1}(\Omega)$, let $v_\varepsilon\in H^m_0(\Omega)$ satisfy $\mathcal{L}^{A}_\varepsilon(v_\varepsilon)=F$ in $\Omega$. By the $W^{m, p}$ estimate for higher-order elliptic systems (\cite{Dong2011_Higher}), there exists a constant $q>2$, depending only on $m, n, \Omega, \mu$, such that,
\begin{align}\|\nabla^m v_\varepsilon\|_{L^q(\Omega)}\leq C\|F\|_{H^{-m+1}(\Omega)},\label{pertur_es_q}\end{align}where $C$ depends only on $m, n, \Omega, \mu$.  Setting $1/p=1/2-1/q$ and choosing $R$ such that  $\Omega\subset B(0, R)$, we  deduce that
\begin{align}
\langle\omega_\varepsilon, F\rangle_{H^{m-1}_0(\Omega)\times H^{-m+1}(\Omega)}&\leq C\|(A(\frac{x}{\varepsilon})-\widetilde{A}(\frac{x}{\varepsilon}))\nabla^mv_\varepsilon\|_{L^2(\Omega)}\|\nabla^m\widetilde{u}_\varepsilon\|_{L^2(\Omega)}\nonumber\\
&\leq C\|A(\frac{x}{\varepsilon})-\widetilde{A}(\frac{x}{\varepsilon})\|_{L^p(\Omega)}\|F\|_{H^{-m+1}(\Omega)}\|\nabla^m\widetilde{u}_\varepsilon\|_{L^2(\Omega)}\nonumber\\
&\leq C\bigg(\fint_{B(0, R/\varepsilon)}|A-\widetilde{A}|^p dx\bigg)^{{1/p}}\|F\|_{H^{-m+1}(\Omega)}\|\nabla^m\widetilde{u}_\varepsilon\|_{L^2(\Omega)},\label{pertur_es_dual}
\end{align}
where \eqref{pertur_es_q} is used in the second step.
Moreover, since $\|\nabla^m\widetilde{u}_\varepsilon\|_{L^2(\Omega)}\leq C\|\nabla^m u_0\|_{L^2(\Omega)}$ with $C$ depending only on $d, m, n, \mu$, it follows from \eqref{pertur_es_dual} that
\begin{align}\label{niu938}
\|\omega_\varepsilon\|_{H^{m-1}_0(\Omega)}\leq C\bigg(\fint_{B(0, R/\varepsilon)}|A-\widetilde{A}|^p dx\bigg)^{{1/p}}\|\nabla^m u_0\|_{L^2(\Omega)}\leq C\varepsilon\|\nabla^m u_0\|_{L^2(\Omega)},
\end{align}
where $C$ depends only on $m, n, \Omega, \mu$.  Since $u_0=\widetilde{u}_0$, estimate \eqref{niu938}, together with the condition that $A$ has the $O(\varepsilon)$-convergence property, yields the desire property for $\widetilde{A}$.
\end{proof}


\section*{Acknowledgements}
This work was completed under the supervision of Professor Zhongwei Shen, to whom the authors are much obliged for the guidance. Special thanks also go to Professor  Russell Brown and University of Kentucky for the warm hospitality and support.
\bibliographystyle{amsplain}
\bibliography{bib}

\providecommand{\bysame}{\leavevmode\hbox to3em{\hrulefill}\thinspace}
\providecommand{\MR}{\relax\ifhmode\unskip\space\fi MR }
\providecommand{\MRhref}[2]{%
  \href{http://www.ams.org/mathscinet-getitem?mr=#1}{#2}
}
\providecommand{\href}[2]{#2}
\begin{thebibliography}{10}

\bibitem{Armstrong2016_Bounded}
S.~N. Armstrong, A.~Gloria, and T.~Kuusi, \emph{Bounded correctors in almost
  periodic homogenization}, Arch. Ration. Mech. Anal. \textbf{222} (2016),
  no.~1, 393--426.

\bibitem{Armstrong2016_Lipschitz}
S.~N. Armstrong and Z.~Shen, \emph{Lipschitz estimates in almost-periodic
  homogenization}, Comm. Pure Appl. Math. \textbf{69} (2016), no.~10,
  1882--1923.

\bibitem{Avellaneda1987_Compactness}
M.~Avellaneda and F.~Lin, \emph{Compactness methods in the theory of
  homogenization}, Comm. Pure Appl. Math. \textbf{40} (1987), no.~6, 803--847.

\bibitem{Barton2016_Gradient}
A.~Barton, \emph{Gradient estimates and the fundamental solution for
  higher-order elliptic systems with rough coefficients}, Manuscripta Math.
  \textbf{151} (2016), no.~3-4, 375--418.

\bibitem{Besicovitch1955}
A.~S. Besicovitch, \emph{Almost periodic functions}, Dover Publications, Inc.,
  New York, 1955.

\bibitem{Bondarenko2005}
A.~Bondarenko, G.~Bouchitt\'e, L.~Mascarenhas, and R.~Mahadevan, \emph{Rate of
  convergence for correctors in almost periodic homogenization}, Discrete
  Contin. Dyn. Syst. \textbf{13} (2005), no.~2, 503--514.

\bibitem{Boyling1996}
J.~B. Boyling, \emph{Green's functions for polynomials in the {L}aplacian}, Z.
  Angew. Math. Phys. \textbf{47} (1996), no.~3, 485--492.

\bibitem{Caffarelli2010}
L.~A. Caffarelli and P.~E. Souganidis, \emph{Rates of convergence for the
  homogenization of fully nonlinear uniformly elliptic pde in random media},
  Invent. Math. \textbf{180} (2010), no.~2, 301--360.

\bibitem{Dong2011_Higher}
H.~Dong and D.~Kim, \emph{Higher order elliptic and parabolic systems with
  variably partially {BMO} coefficients in regular and irregular domains}, J.
  Funct. Anal. \textbf{261} (2011), no.~11, 3279--3327.

\bibitem{Dungey2001}
N.~Dungey, A.~F.~M. ter Elst, and D.~W. Robinson, \emph{On second-order
  almost-periodic elliptic operators}, J. London Math. Soc. (2) \textbf{63}
  (2001), no.~3, 735--753.

\bibitem{Erdelyi1953}
A.~Erd\'elyi, W.~Magnus, F.~Oberhettinger, and F.~G. Tricomi, \emph{Higher
  transcendental functions. vol. {II}}, McGraw-Hill Book Company, Inc., New
  York-Toronto-London, 1953.

\bibitem{Jikov1994}
V.~V. Jikov, S.~M. Kozlov, and O.~A. Ole\u{\i}nik, \emph{Homogenization of
  differential operators and integral functionals}, Springer Berlin Heidelberg,
  1994.

\bibitem{Kozlov1979}
S.~M. Kozlov, \emph{Averaging of differential operators with almost periodic,
  rapidly oscillating coefficients}, Math. USSR-Sbornik \textbf{35} (1979),
  no.~4, 481--498.

\bibitem{Kukushkin2016}
A.~A. Kukushkin and T.~A. Suslina, \emph{Homogenization of high-order elliptic
  operators with periodic coefficients}, Algebra i Analiz \textbf{28} (2016),
  no.~1, 89--149.

\bibitem{Lieb2001}
E.~H. Lieb and M.~Loss, \emph{Analysis}, second ed., Graduate Studies in
  Mathematics, vol.~14, American Mathematical Society, Providence, RI, 2001.

\bibitem{Lions2005}
P.~L. Lions and P.~E. Souganidis, \emph{Homogenization of degenerate
  second-order {PDE} in periodic and almost periodic environments and
  applications}, Ann. Inst. H. Poincar\'{e} Anal. Non Lin\'{e}aire \textbf{22}
  (2005), no.~5, 667--677.

\bibitem{Niu2018_Convergence}
W.~Niu, Z.~Shen, and Y.~Xu, \emph{Convergence rates and interior estimates in
  homogenization of higher order elliptic systems}, J. Funct. Anal.
  \textbf{274} (2018), no.~8, 2356--2398.

\bibitem{Niu2019_Boundary}
W.~Niu and Y.~Xu, \emph{Uniform boundary estimates in homogenization of
  higher-order elliptic systems}, Ann. Mat. Pura Appl. (4) \textbf{198} (2019),
  no.~1, 97--128.

\bibitem{Pastukhova2016}
S.~E. Pastukhova, \emph{Estimates in homogenization of higher-order elliptic
  operators}, Appl. Anal. \textbf{95} (2016), no.~7, 1449--1466.

\bibitem{Pastukhova2017}
\bysame, \emph{Operator error estimates for homogenization of fourth order
  elliptic equations}, St. Petersburg Math. J. \textbf{28} (2017), no.~2,
  273--289.

\bibitem{Pozhidaev1989}
A.~V. Pozhidaev and V.~V. Yurinski\u\i, \emph{On the error of averaging of
  symmetric elliptic systems}, Izv. Akad. Nauk SSSR Ser. Mat. \textbf{53}
  (1989), no.~4, 851--867, 912.

\bibitem{Shen2015_Convergence}
Z.~Shen, \emph{Convergence rates and {H}\"older estimates in almost-periodic
  homogenization of elliptic systems}, Anal. PDE \textbf{8} (2015), no.~7,
  1565--1601.

\bibitem{Shen2017_Boundary}
\bysame, \emph{Boundary estimates in elliptic homogenization}, Anal. PDE
  \textbf{10} (2017), no.~3, 653--694.

\bibitem{Shen2017_Convergence}
Z.~Shen and J.~Zhuge, \emph{Convergence rates in periodic homogenization of
  systems of elasticity}, Proc. Amer. Math. Soc. \textbf{145} (2017), no.~3,
  1187--1202.

\bibitem{Shen2018_Approximate}
\bysame, \emph{Approximate correctors and convergence rates in almost-periodic
  homogenization}, J. Math. Pures Appl. (9) \textbf{110} (2018), 187--238.

\bibitem{Suslina2013_Dirichlet}
T.~A. Suslina, \emph{Homogenization of the {D}irichlet problem for elliptic
  systems: {$L^2$}-operator error estimates}, Mathematika \textbf{59} (2013),
  no.~2, 463--476.

\bibitem{Suslina2013_Neumann}
\bysame, \emph{Homogenization of the {N}eumann problem for elliptic systems
  with periodic coefficients}, SIAM J. Math. Anal. \textbf{45} (2013), no.~6,
  3453--3493.

\bibitem{Suslina2017_Dirichlet}
\bysame, \emph{Homogenization of the {D}irichlet problem for higher-order
  elliptic equations with periodic coefficients}, Algebra i Analiz \textbf{29}
  (2017), no.~2, 139--192.

\bibitem{Suslina2018_Neumann}
\bysame, \emph{Homogenization of the {N}eumann problem for higher order
  elliptic equations with periodic coefficients}, Complex Var. Elliptic Equ.
  \textbf{63} (2018), no.~7-8, 1185--1215.

\bibitem{Zhuge2017_Uniform}
J.~Zhuge, \emph{Uniform boundary regularity in almost-periodic homogenization},
  J. Differential Equations \textbf{262} (2017), no.~1, 418--453.

\end{thebibliography}

\vspace{0.5cm}

\noindent Yao Xu \\
Institute of Mathematics, Academy of Mathematics and Systems Science, Chinese Academy of Sciences, Beijing, 100190, CHINA\\
E-mail: xuyao89@gmail.com\\

\noindent Weisheng Niu \\
School of Mathematical Science, Anhui University\\
Hefei, 230601, CHINA\\
E-mail: weisheng.niu@gmail.com\\

\noindent\today

\end{document}